\documentclass[12pt]{amsart}

\usepackage{amsmath}

\usepackage[numbers]{natbib}
\setcitestyle{square}

\usepackage{graphicx}
\usepackage{amssymb}
\usepackage{tabularx}
\usepackage{subfigure}
\usepackage[export]{adjustbox}
\usepackage{booktabs}
\usepackage{threeparttable}
\usepackage{amsthm}
\usepackage[hidelinks]{hyperref}
\usepackage{cleveref}

\usepackage[margin=1in]{geometry}

\crefname{Theorem}{Theorem}{Theorems}
\crefname{Lemma}{Lemma}{Lemmas}
\crefname{Proposition}{Proposition}{Propositions}
\crefname{Corollary}{Corollary}{Corollaries}
\crefname{Definition}{Definition}{Definitions}
\crefname{Assumption}{Assumption}{Assumptions}
\crefname{Remark}{Remark}{Remarks}
\crefname{Example}{Example}{Examples}

\newtheorem{Definition}{Definition}

\newtheorem{Theorem}{Theorem}

\newtheorem{Lemma}{Lemma}
\newtheorem{Corollary}{Corollary}
\newtheorem{Example}{Example}

\newcommand{\A}{\mathbb{P}} 
\newcommand{\Q}{\mathbb{A}}
\newcommand{\Qu}{\mathbb{A}_u}
\newcommand{\Ql}{\mathbb{A}_l}
\newcommand{\D}[1]{\| #1 \|_{rng}}

\numberwithin{equation}{section} 

\author{Mingxi Li}
\thanks{Beihang University, Beijing, China. Email: \texttt{mingxi\_li@buaa.edu.cn}}

\author{Hanzhou Wang}
\thanks{Beihang University, Beijing, China. Email: \texttt{wanghanzhou@buaa.edu.cn}}

\author{Dongyu Li}
\thanks{Beihang University, Beijing, China; National University of Singapore, Singapore. 
  Corresponding author. Email: \texttt{dongyuli@buaa.edu.cn}}

\title[Signum Consensus Protocol on Arbitrary Weighted Directed Graphs]{From Control to Opinion Dynamics: Signum Consensus Protocol on Arbitrary Weighted Directed Graphs}

\begin{document}

\begin{abstract}
We study the signum consensus protocol in continuous-time systems over arbitrary weighted directed graphs with bounded disturbances. The right-hand side of the differential equation is discontinuous on a codimension-$n$ manifold ($n > 1$). On such a manifold, the Filippov sliding vector is not uniquely determined. This results in non-unique solutions and makes the analysis of the system challenging. We define the Polarization Index as the supremum of the growth rate of the difference between the maximum and minimum agent states in the system, derive its closed-form expression, and show that some solution attains this supremum at all forward times except during consensus. From this result, we derive necessary and sufficient conditions for consensus and provide a least upper bound on consensus time. To address the high computational complexity of evaluating the Polarization Index, we propose a low–average–complexity algorithm. Finally, we develop an opinion dynamics model grounded in the signum consensus protocol, revealing a fundamental link between dissensus and community structures.
\end{abstract}



\maketitle

\section{Introduction}
In multi-agent systems (MASs), the consensus problem aims to drive all agents' states to gradually approach or become equal through local sensing or communication. Let each agent $i$ has a scalar state $x_i: \mathbb{R} \to \mathbb{R}$. The classic linear consensus protocol in \Cref{e:i:1} drives each agent toward the weighted average of its neighbors' states. The signum consensus protocol in \Cref{e:i:2} drives each agent toward the side with a higher weighted number of neighbors.
\begin{equation}\label{e:i:1}
    \dot{x}_i(t) = \sum_{j\in\mathcal{V}}w_{ji}\big(x_j(t)-x_i(t)\big), \qquad i\in \mathcal{V},
\end{equation}
\begin{equation}\label{e:i:2}
    \dot{x}_i(t) = \sum_{j\in\mathcal{V}}w_{ji}\operatorname{sign}\big(x_j(t)-x_i(t)\big), \qquad i\in \mathcal{V},
\end{equation}
where $\mathcal{V}$ is the agent set, $w_{ji}\in \mathbb{R}_{\ge 0}$ denotes the influence strength from agent $j$ to $i$.

The only difference between \Cref{e:i:1} and \Cref{e:i:2} is the $\operatorname{sign}(\cdot)$ function but this difference brings many advantages such as robustness to disturbances, 
resistance to attacks 
\cite{santilli2021dynamic,franceschelli2016finite}, low communication and sensing requirements \cite{chen2011finite}, and finite-time consensus \cite{chen2011finite,liu2015finite}. 

However, the right-hand side of \Cref{e:i:2} is discontinuous on a codimension-$n$ manifold ($n > 1$), the Filippov sliding vector on it cannot be uniquely determined \cite{dieci2011sliding, jeffrey2014dynamics}, which leads to non-unique solutions and makes the analysis of the system challenging. Existing studies on \Cref{e:i:2} and related protocols 
\cite{chen2011finite,franceschelli2016finite,liu2015finite,santilli2021dynamic} typically assume undirected or detail-balanced communication graphs, which lead to the uniqueness of the solution and thereby simplify the analysis. No consensus analysis has been done for \Cref{e:i:2} on arbitrary weighted directed graphs.

Nevertheless, analyzing \Cref{e:i:2} on arbitrary weighted directed graphs is necessary. MAS may face attacks, interference, or device failures that disrupt the communication graph structure. For example, communication failures may prevent some agents from sending messages while still receiving them, turning an undirected graph into a directed one. Actuator errors can also reduce controller efficiency and break the detailed balance condition. The analysis of \Cref{e:i:2} on arbitrary weighted directed graphs provides a theoretical foundation for network redundancy design and fault analysis. Moreover, directional communication is essential in practice, as some scenarios can only be modeled using directed graphs, such as broadcasting.

Its intriguing properties on arbitrary weighted directed graphs also suggest that studying \Cref{e:i:2} in such settings is worthwhile. Our study of \Cref{e:i:2} unexpectedly led to indices related to community structure \cite{fortunato2016community}, where communities are groups of nodes more densely connected internally than externally. Identifying communities reveals network organization and highlights regions with partial autonomy. We also found that \Cref{e:i:2} can serve as a model for opinion dynamics, where many of its properties become naturally interpretable. Individuals’ opinions represent their cognitive orientations toward certain objects (e.g., specific issues, events, or other individuals), for example, displayed attitudes or subjective certainties of belief.
Opinion dynamics models how individual opinions evolve through social interactions in a networked population 
\cite{proskurnikov2017tutorial}. While natural and engineered networks often exhibit spontaneous order, social communities display more complex, irregular dynamics. Opinion dynamics research thus calls for theories that can explain the emergence of both agreement (i.e., consensus) and disagreement (i.e., dissensus) \cite{abelson1967mathematical,horowitz1962consensus,proskurnikov2017tutorial}, a challenge known as the community cleavage problem or Abelson’s problem \cite{friedkin2015problem}. Compared to models that require additional assumptions, such as the presence of stubborn agents 
\cite{tian2021social}, quantized state/behavior \cite{ceragioli2011discontinuities,ceragioli2018consensus,ceragioli2025disruptive}, or bounded confidence between individuals \cite{rainer2002opinion,blondel2010continuous,brooks2024emergence,motsch2014heterophilious,jabin2014clustering}, \Cref{e:i:2}, as one of the simplest discontinuous protocols, captures both agreement and disagreement without relying on assumptions specifically designed for opinion dynamics. 
The main assumption, modeled by the $\operatorname{sign}(\cdot)$ function in \Cref{e:i:2}, is that individuals can only compare whether others hold more radical or more moderate opinions on an issue, without knowing the exact opinions of others. This is reasonable, as individuals may avoid fully expressing their views to prevent conflict or being seen as outliers. This setting resembles cases in which the observable outcome is restricted to a binary response (``approve'' or ``disapprove''), providing little indication of the underlying strength or intensity of the opinion. We believe that \Cref{e:i:2} satisfies the criteria outlined in \cite{proskurnikov2017tutorial} for being sufficiently ``rich'' to capture the behavior of social actors, while also being ``simple'' enough to be rigorously analyzed.

This paper analyzes the consensus behavior of all possible solutions of the signum consensus protocol on arbitrary weighted directed graphs in the presence of disturbances. Due to the discontinuity of the right-hand side of the differential equation, classical solutions may not exist. We adopt the definition of solution proposed in \cite[p. 54]{filippov2013differential} (termed parametric combinations in \cite{jeffrey2014dynamics}) to characterize mutually independent evaluations of each sign function and and disturbances at points of discontinuity, which encompasses situations in which different sign functions are implemented by different relays, whose manufacturing imperfections may cause inconsistencies in their outputs near zero. To resolve the ambiguity caused by solution non-uniqueness, we distinguish between two notions: Strong Consensus, where every solution achieves consensus for each initial condition; and Weak Consensus, where at least one consensus solution exists for each initial condition. We describe our contribution as twofold:

As the \textit{first} contribution, we present the necessary and sufficient conditions for the system to achieve Strong Consensus and provide a least upper bound on consensus time. We define the Polarization Index $\A$ as the supremum of the growth rate of the difference between the maximum and minimum agent states in the system, derive its closed-form expression, and show that some solution attains this supremum at all forward times except during consensus. Therefore, $\A < 0$ is both necessary and sufficient for the system to achieve Strong Consensus. A least upper bound on the consensus time is also established. We reformulate the computation of $\A$ as a mixed-integer programming program in certain cases, reducing average-case complexity.

As the \textit{second} contribution, we develop an opinion dynamics model grounded in the signum consensus protocol that uncovers the fundamental link between dissensus and community structures. We show that the absence of Strong Communities is a sufficient condition for Strong Consensus, while the absence of Satisfactory Partitions \cite{bazgan2006satisfactory} is a necessary condition. Moreover, we find that the Autonomy Index $\Q(\cdot)$ proposed in this paper serves as a measure of community strength, while the Polarization Index $\A$ quantifies the quality of the best partition within a graph.

This paper is organized as follows:
\Cref{s:Dynamic Modeling} defines the studied system.
\Cref{s:Existence of Solutions} introduces the adopted solution concept, with \Cref{t:m:1} proving the existence of solutions.
\Cref{s:Nonuniqueness of Solutions} discusses solution non-uniqueness.
\Cref{s:Consensus Analysis} analyzes consensus, with \Cref{t:m:3} providing properties related to the difference between the maximum and minimum agent states in the system and \Cref{t:m:4} giving necessary and sufficient conditions for strong consensus and a least upper bound on consensus time.
\Cref{s:Computation of A} presents a lower–average–complexity algorithm for computing $\A$, where \Cref{t:m:5} reformulates it as a max–min integer program and \Cref{t:m:8} reduces it to a mixed-integer program under certain conditions.
\Cref{s:Modeling of Opinion Dynamics and Interpretation of Notation} develops an opinion dynamics model, with \Cref{t:sufficient condition} and \Cref{t:necessary condition} linking strong communities and satisfactory partitions to dissensus. \Cref{s:Illustrative Examples} gives some simulation results. \Cref{s:Conclusion} gives a brief conclusion.

\subsection*{\textbf{Notation}}

Let $\mathbb{R}$, $\mathbb{R}_{\ge 0}$, and $\overline{\mathbb{R}} := \{-\infty,\infty\} \cup \mathbb{R}$  denote the set of real numbers, the set of nonnegative real numbers, and the set of extended real numbers, respectively. The bold symbols $\pmb{0}$ and $\pmb{1}$ denote vectors of appropriate dimensions. For two vectors $a, b \in \mathbb{R}^d$, $a \succeq b$ (equivalently, $b \preceq a$) implies that $a_i \geq b_i$ for all $i \in \{1, \dots, d\}$. 
The set difference of two sets $A$ and $B$ is denoted $A\backslash B$. The notation $\operatorname{card}(S)$ and $\mathcal{P}(S)$ represent the cardinality and the power set of $S$, respectively. $\operatorname{conv}(\cdot)$ denotes convex hull. $\sup(\cdot)$ represents the supremum and we formally write $\sup \emptyset = -\infty$. $|\cdot|: \mathbb{R} \to \mathbb{R}_{\ge 0}$ denotes the absolute value function. $\operatorname{span}(v)$ denotes the subspace generated by $v\in \mathbb{R}^{k}$. The identity matrix is denoted by $E$. 

If $I$ is an index set and $a \in I$ (or $A \subset I$), the notation $(\cdot)_a$ (or $(\cdot)_A$) refers to the value (or values) corresponding to the dimension indexed by $a$ (or $A$). For example, if $v\in \mathbb{R}^3$, $S\subset \mathbb{R}^3$, $f:\mathbb{R}\to \mathbb{R}^3$, and $A = \{1,3\}$ then $v_1$ denote the first component of $v$, $v_A = [v_1,v_3]^T$, $S_1 = \{v_1\mid v\in S\}$, and $f_1:\mathbb{R}\to \mathbb{R}^3,x\mapsto (f(x))_1$. The originally subscripted notation is represented using superscripts, such as $t^0$ denoting a specific moment, typically the initial time. It should be noted that, if otherwise specified, the subscript may not adhere to the aforementioned meaning.

\section{Dynamic Modeling}\label{s:Dynamic Modeling}
Consider a network of \( n \) agents (\( n \geq 2 \)) whose communication graph is modeled by a weighted directed graph \( \mathcal{G} = (\mathcal{V}, \mathcal{E}, \mathcal{W}) \), with agent set \( \mathcal{V} = \{1, \dots, n\} \), edge set \( \mathcal{E} \subset \mathcal{V} \times \mathcal{V} \), and weight matrix \( \mathcal{W} = [w_{ij}] \in \mathbb{R}_{\geq 0}^{n \times n} \). Each edge $(j, i) \in \mathcal{E}$ means agent $i$ can compare its state with agent $j$'s, even without knowing their exact values. $w_{ji} > 0$ if and only if $(j, i) \in \mathcal{E}$. The diagonal element $w_{ii}$ (self-loop) carries no useful information; We use it to represent the upper bound of the disturbance term acting on agent $i$.

Each agent $i$ holds a scalar state $x_i(t) \in \mathbb{R}$ with continuous-time dynamics:
\begin{equation}\label{e:c:3}
    \dot{x}_i(t) = d_i(t,x(t)) + \sum_{j\in \mathcal{V}} w_{ji} \operatorname{sign}(x_j(t)-x_i(t))\quad \forall i\in \mathcal{V},
\end{equation}
where $d: \mathbb{R} \times \mathbb{R}^n \to \mathbb{R}^n$ models external disturbances and actuator errors, bounded componentwise by
\begin{equation}\label{d satisfies}
    -w_{ii} \le d_i(t,x) \le w_{ii} \quad \forall (t,x) \in \mathbb{R}^{n+1},\ i \in \mathcal{V}.
\end{equation}

\section{Existence of Solutions}\label{s:Existence of Solutions}

For $\dot{x} = f(t, x)$ with discontinuous $f:\mathbb{R} \times \mathbb{R}^n \to \mathbb{R}^n$, classical solutions may not exist, requiring a generalized notion of solution. One approach is to define an absolutely continuous function $x:\mathbb{R}\to\mathbb{R}^n$ satisfying differential inclusion
\begin{equation}
    \dot{x}(t) \in \mathcal{F}(t, x(t))
\end{equation}
as a solution to the original equation, where $\mathcal{F}(t,x) = \{f(t,x)\}$ when $f$ is continuous, and otherwise follows definitions like the simplest convex definition \cite{filippov2013differential}, the equivalent control definition \cite{utkin2003variable}, or the general definition from \cite{aizerman1974fundamentals}. 

There is no universally ``correct'' definition; rather, different definitions are appropriate in different contexts. As Filippov noted in \cite[p. 1]{filippov2013differential}, when $f$ is discontinuous in $x$, simple mathematical arguments often fail, and solutions are defined via a limiting process that reflects the system’s physical meaning. For example, static friction at zero velocity cannot be inferred from dynamic friction at velocities near zero and must be measured directly \cite[p. 53]{filippov2013differential}.

Since $\operatorname{sign}(x_j - x_i)$ and $\operatorname{sign}(x_i - x_j)$ are computed independently by agents $i$ and $j$, we consider the possibility that, due to implementation errors, they may have the same sign when $|x_i - x_j|$ is very small. In opinion dynamics, such discrepancies can also arise from perception errors, where both individuals may believe the other holds a more extreme opinion. Whether such errors are considered can greatly affect the value of $\mathcal{F}(t,x)$ at discontinuity points (see \cite[p. 53]{filippov2013differential} for example).

Therefore, we adopt the definition of solution proposed in \cite[p. 54]{filippov2013differential} (termed parametric combinations in \cite{jeffrey2014dynamics}). Consider that the values of different sign functions $\operatorname{sign}(\cdot)$ and disturbances $d_i(\cdot)$ are mutually independent near discontinuity points. Rewrite \Cref{e:c:3} as
\begin{equation}\label{e:c:3.1}
    \dot{x}(t) = f(d_1(t,x(t)),\dots,d_n(t,x(t)),u_{11}(x(t)),u_{12}(x(t)),\dots,u_{nn}(x(t))),
\end{equation}
where $f$ is the stack of $f_i(d_1,\dots,d_n,u_{11},\dots,u_{nn}) := d_i+ \sum_{j\in \mathcal{V}} u_{ji}$ and $u_{ji}(x) := w_{ji}\operatorname{sign}(x_j-x_i)$.  
Define the set-valued functions $\mathcal{D}_i:\mathbb{R}\times\mathbb{R}^n\to  \mathcal{P}(\mathbb{R})$, $\mathcal{U}_{ij}:\mathbb{R}^n\to  \mathcal{P}(\mathbb{R})$, and $\mathcal{F}:\mathbb{R}\times\mathbb{R}^n\to\mathcal{P}(\mathbb{R}^n)$ as
\begin{equation}\label{e:s:2}
    \begin{aligned}
        &\begin{aligned}
        \mathcal{D}_i(t,x) := \operatorname{conv}\{&\lim_{k\to \infty} d_i(t^k,x^k) \mid (t^k,x^k)\to  (t,x), k = 1,2,\dots \},
        \end{aligned}
        \\
        &\mathcal{U}_{ij}(x) := \operatorname{conv}\{\lim_{k\to \infty} u_{ij}(x^k) \mid x^k\to  x, k = 1,2,\dots \},\\
        &\mathcal{F}(t,x) := \operatorname{conv}f(\mathcal{D}_1(t,x),\dots,\mathcal{D}_n(t,x),\mathcal{U}_{11}(x),\mathcal{U}_{12}(x),\dots,\mathcal{U}_{nn}(x)).
    \end{aligned}
\end{equation}

Then the Filippov Solution of \Cref{e:c:3} is defined as below.
\begin{Definition}[Filippov Solution \cite{filippov2013differential}]
    $x:[t^0,t^1]\rightarrow \mathbb{R}^n$ is a Filippov Solution of \Cref{e:c:3} if $x(\cdot)$ is absolutely continuous, $\dot{x}(t)\in \mathcal{F}(t,x(t))$ a.e. on $[t^0,t^1]$, and $x(t^0) = x^0$ where $(t^0,x^0)$ is the initial condition.
\end{Definition}

\begin{Theorem}\label{t:m:1}
    For any $\mathcal{W}\in \mathbb{R}_{\ge 0}^{n\times n}$, $d(\cdot,\cdot)$ satisfying \Cref{d satisfies}, and initial condition $(t^0,x^0)\in \mathbb{R}^{n+1}$, there exists a solution $x:\mathbb{R}\to \mathbb{R}^n$ to \Cref{e:c:3} with $x(t^0) = x^0$.
\end{Theorem}
\begin{proof}
    See \Cref{Proof of t:m:1}
\end{proof}

\section{Non-uniqueness of Solutions}\label{s:Nonuniqueness of Solutions}
We provide examples to illustrate the non-uniqueness of solutions, analyze its causes, and discuss its impact on the system.

A solution $x(\cdot)$ of \Cref{e:c:3} is called a consensus solution \cite{wang2010finite} if
\begin{equation}
    \lim_{t \to \infty} |x_i(t) - x_j(t)| = 0 \quad \forall\, i,j \in \mathcal{V}.
\end{equation}

\begin{Example}\label{Example 2}
Let $n = 3$, $\mathcal{W} = \begin{bmatrix}
    0\ 2\ 0;
    3\ 0\ 1; 
    1\ 0\ 0
\end{bmatrix}$, and $d(t,x)\equiv \pmb{0}$. 

Define $x,x':\mathbb{R}_{\ge0}\to \mathbb{R}^3$,
\begin{equation}
\begin{aligned}
    x(t) &= \begin{cases}
        [4t,1-2t,2-t]^T & 0\le t < 1/6,\\
        [3/5+2/5t,3/5+2/5t,2-t]^T & 1/6 \le t < 1,\\
        [1,1,1]^T & 1 \le t,
    \end{cases}\\
    x'(t) &= \begin{cases}
        [4t,1-2t,2-t]^T & 0\le t < 1/6,\\
        [1-2t,1-2t,2-t]^T & 1/6 \le t.
    \end{cases}
\end{aligned}
\end{equation}

The graph structure is shown in \Cref{f:s:3a}, and the $x-t$ diagrams of two Filippov solutions $x(\cdot)$ (left) and $x'(\cdot)$ (right) with the same initial condition are shown in \Cref{f:s:3b} (see \Cref{Proof of Example 2} for proof). Here, $x(\cdot)$ reaches consensus, while $x'(\cdot)$ does not.

\begin{figure}[htbp]
	\centering  
	\subfigure[Graph structure]{
		\raisebox{0.2\height}{\label{f:s:3a}\includegraphics[width=0.27\linewidth]{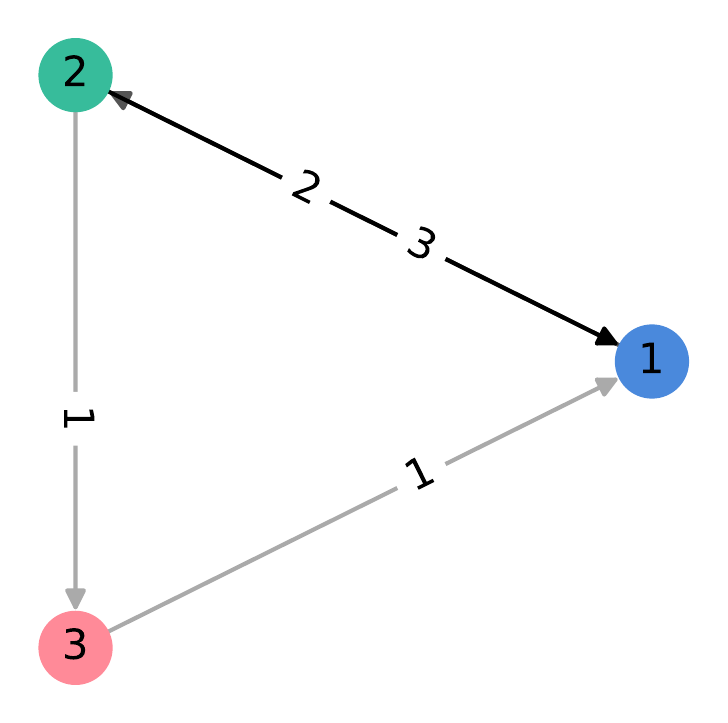}}}
    \subfigure[$x-t$ diagram]{
		\label{f:s:3b}\includegraphics[width=0.66\linewidth]{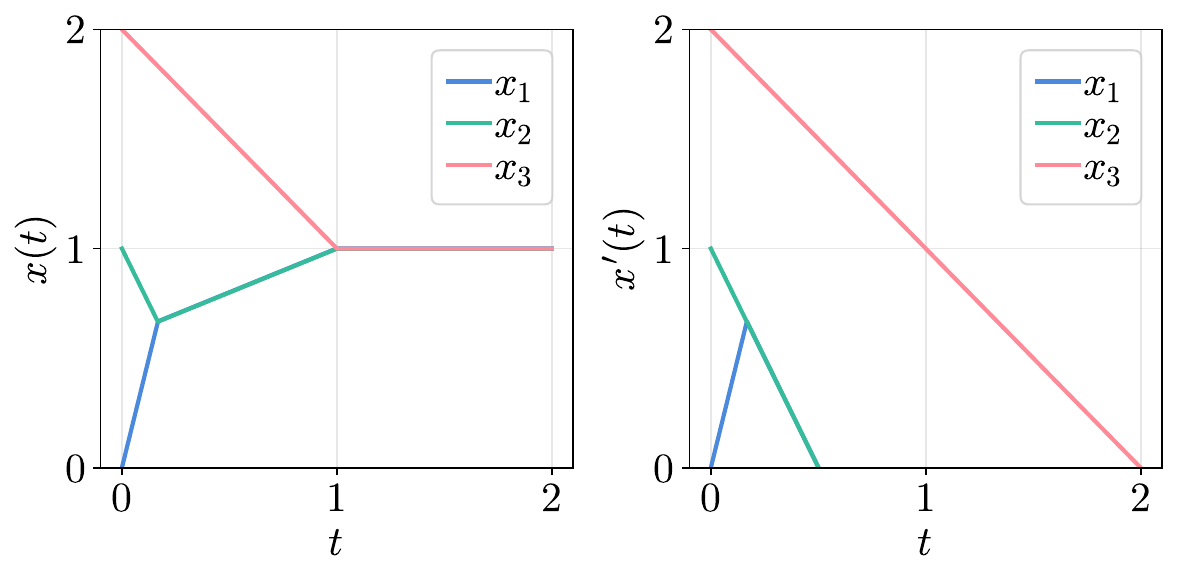}}
    \caption{Graph structure and $x-t$ diagrams of two solution of \Cref{e:c:3} with the same initial condition $x(0) = [0,1,2]^T$}
\end{figure}
\end{Example}

The above example directly provides non-unique solutions, which we now illustrate through simulation. Since simulations yield only a single solution, the non-uniqueness manifests as sensitivity of the system's evolution to initial conditions.

\begin{Example}\label{Example 1}
    Let $n = 6$, with a strongly connected graph $\mathcal{W}$ illustrated in \Cref{Example 1 G} and $d(t,x)\equiv \pmb{0}$. The simulation results are shown in \Cref{fig Example 1} (see caption for details).
    \begin{figure}[htbp]
    	\centering  
    	\subfigure[Graph structure]{\raisebox{0.07\height}{\label{Example 1 G}\includegraphics[width=0.3\linewidth]{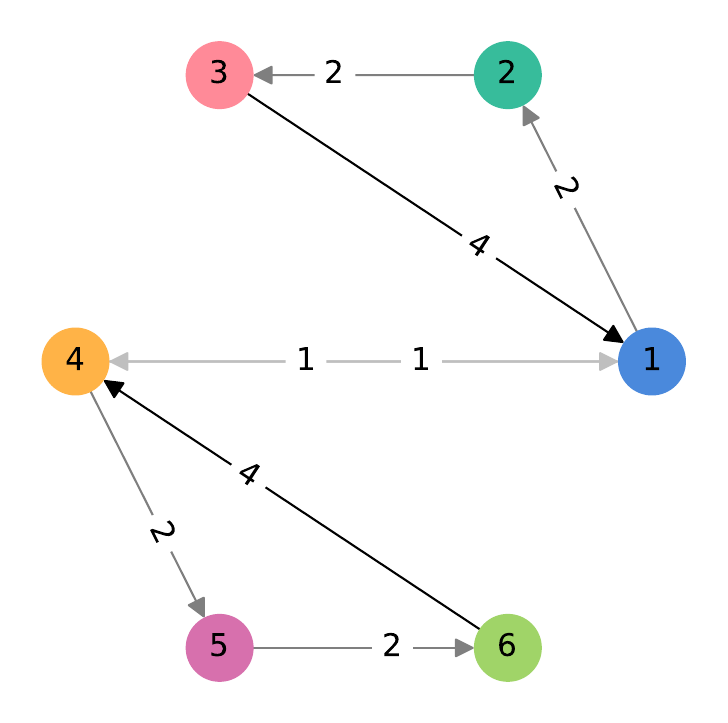}}}
    	\subfigure[{Simulation with $x(0) = [0, 1, 2, 4, 3, 2]^T$. Left: full trajectories; right: zoom-in on region R1, marked in the left plot.}]{\label{Example 1 xt 1}\includegraphics[width=0.305\linewidth]{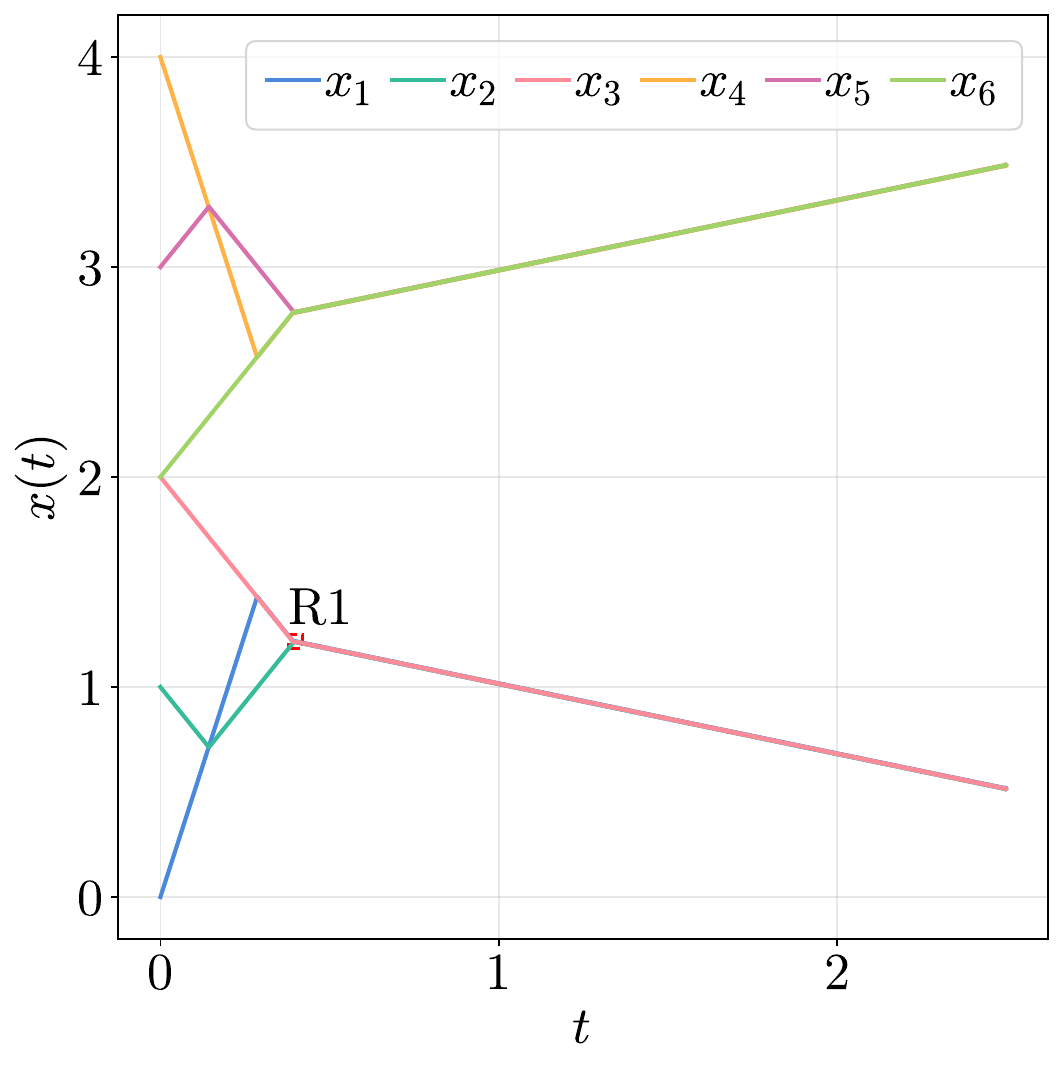}\hspace{0.01\linewidth}
        \includegraphics[width=0.33\linewidth]{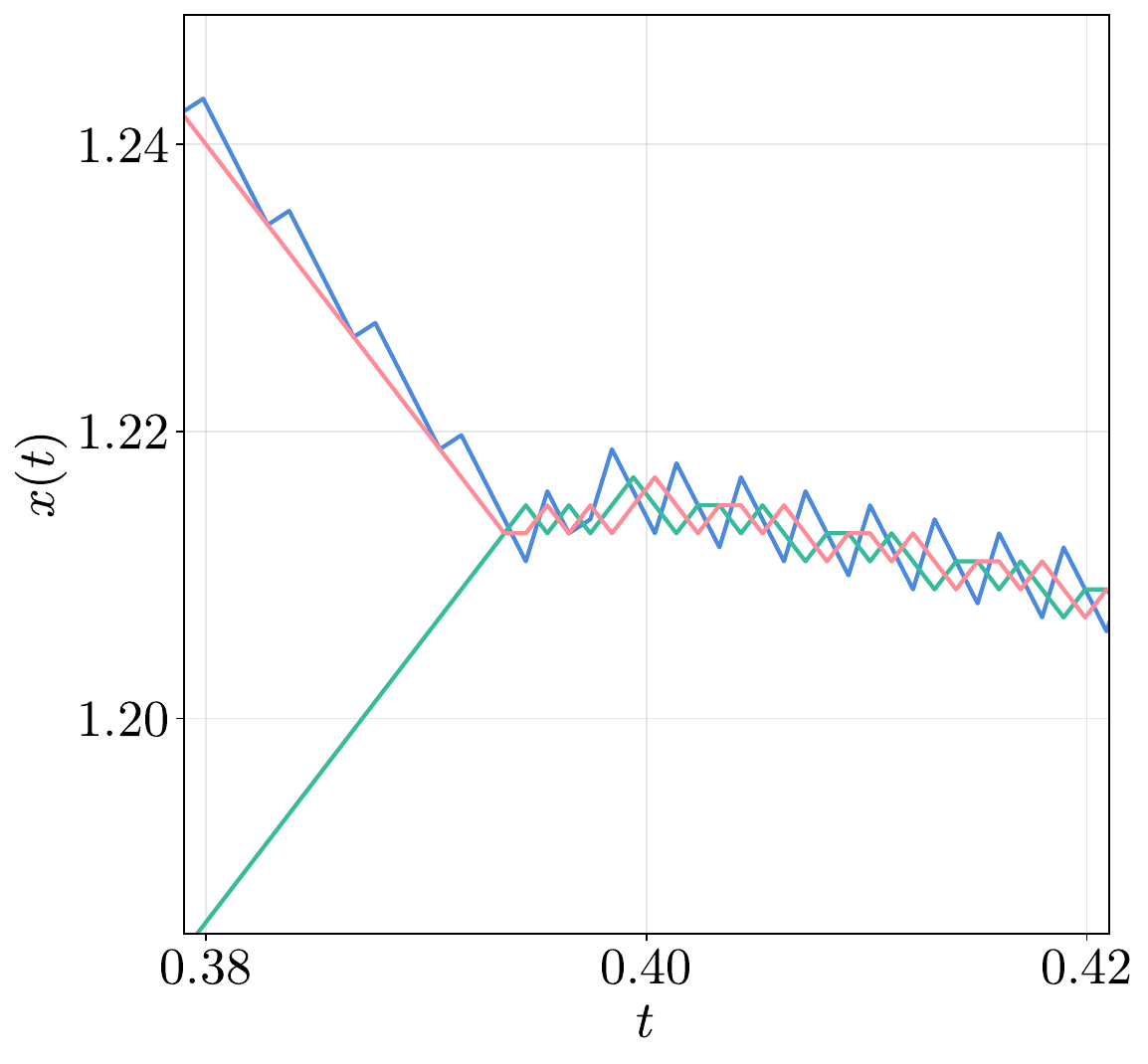}}
        \subfigure[{Simulation with $x(0) = [0.001, 1.0001, 2, 3.999, 3, 2]^T$. Left: full trajectories; middle and right: zoom-ins on regions R2 and R3, marked in the left plot.}]{\label{Example 1 xt 2}\includegraphics[width=0.305\linewidth]{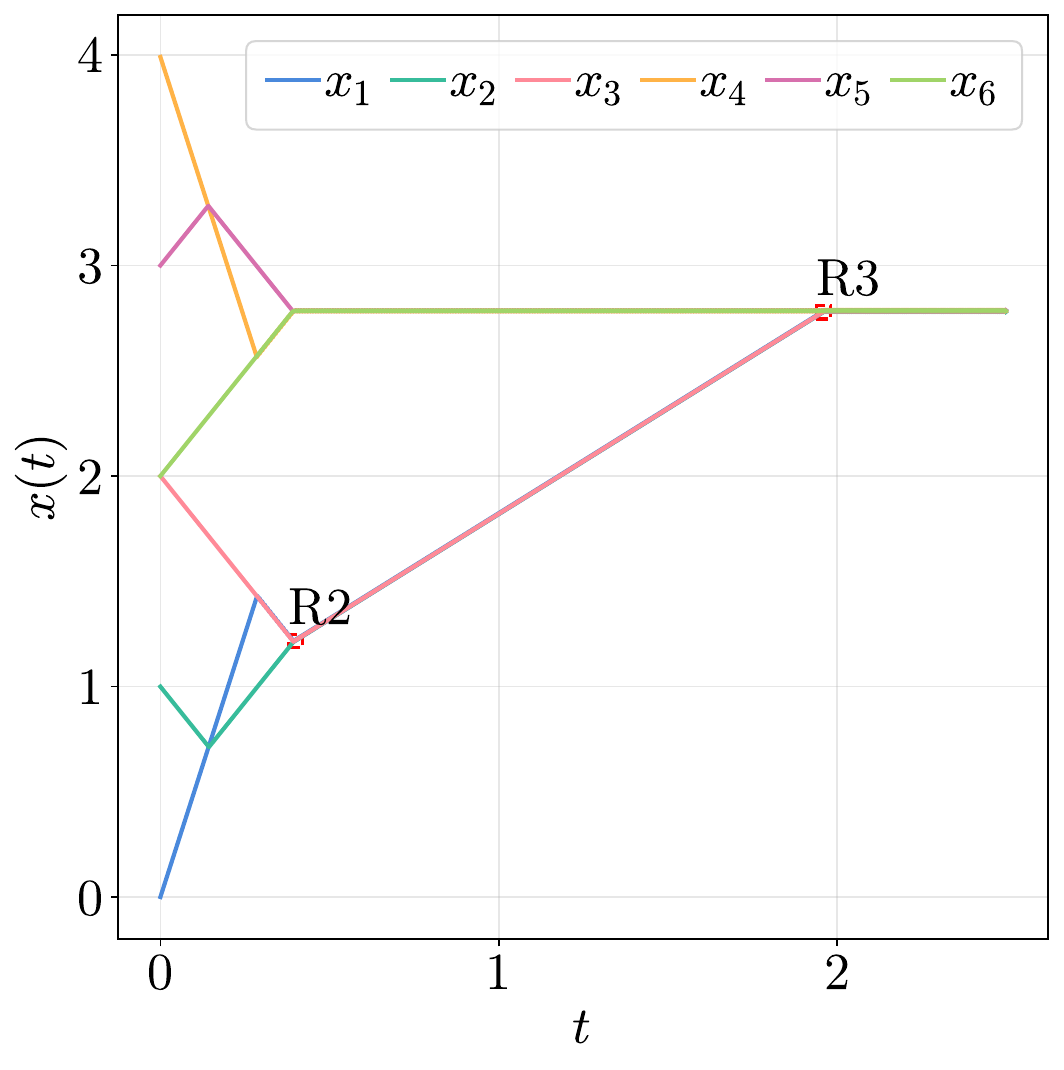}\includegraphics[width=0.33\linewidth]{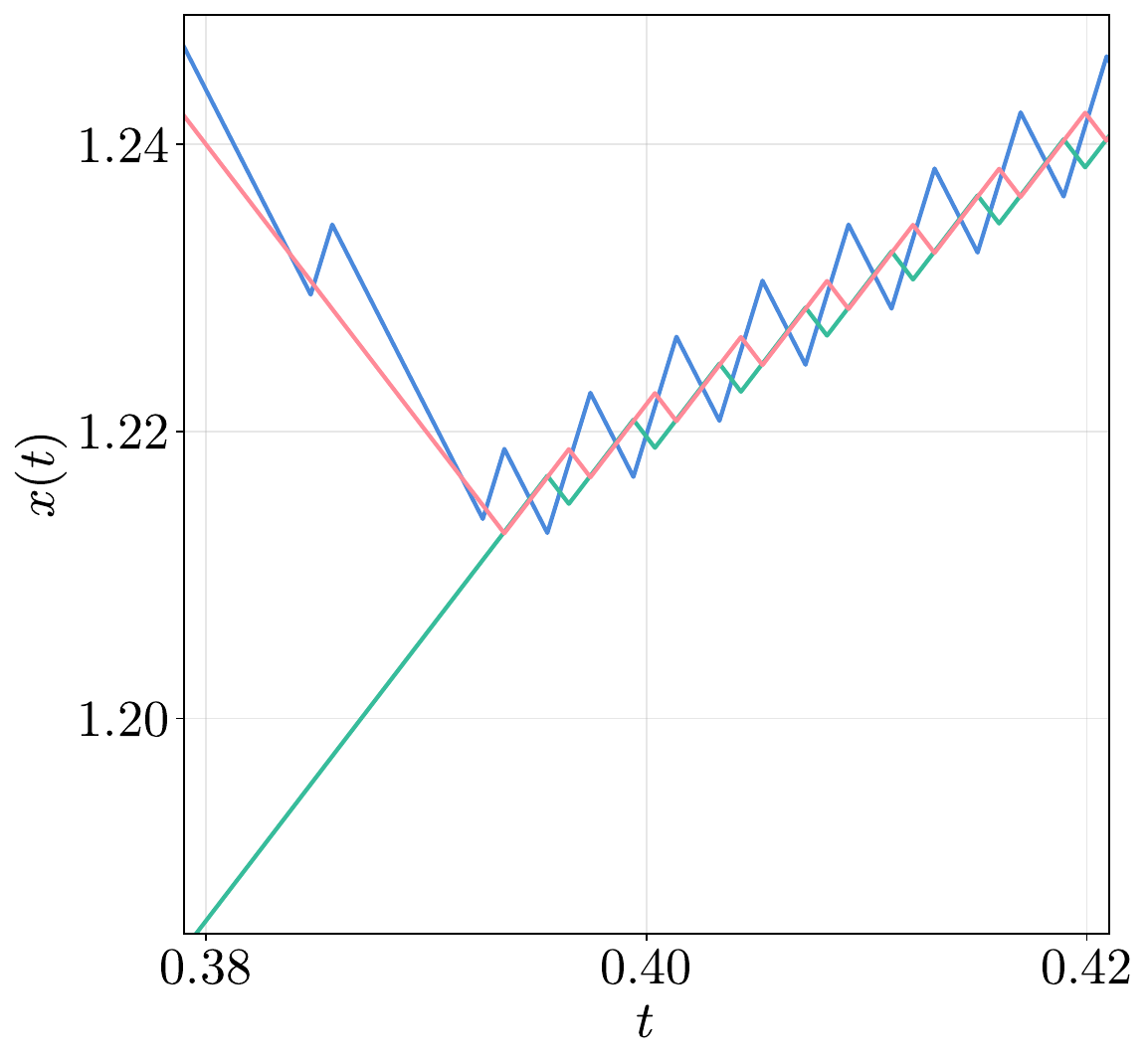}\includegraphics[width=0.33\linewidth]{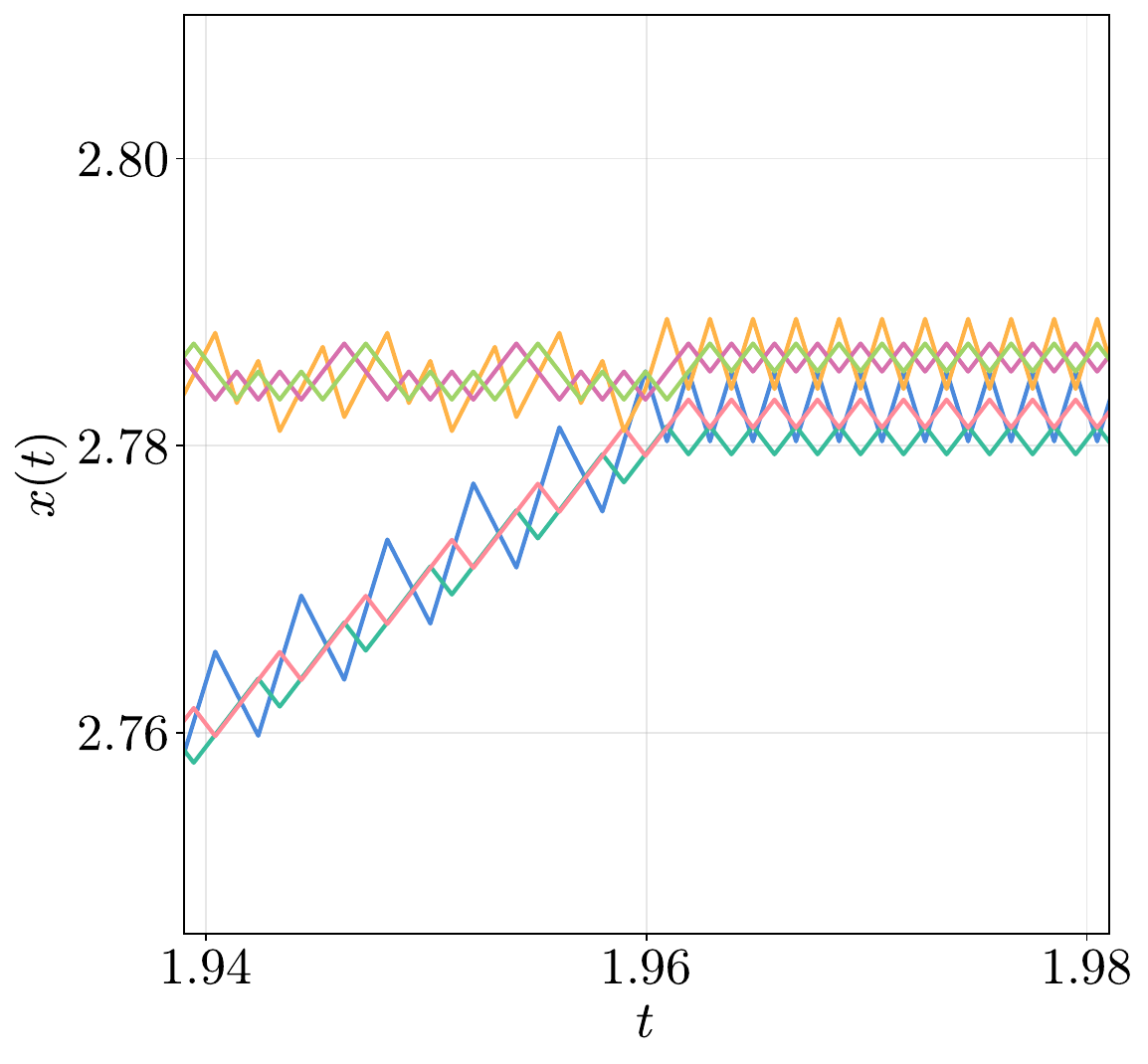}}
        \caption{Graph structure and simulation results under two different initial conditions in \Cref{Example 1}. Simulations used the Simple Euler method with a time step of $2^{-10}$ s. Both dissensus (\Cref{Example 1 xt 1}) and consensus solutions (\Cref{Example 1 xt 2}) are shown. Experiments show that two nearly identical initial conditions can lead to drastically different outcomes.
        }
        \label{fig Example 1}
    \end{figure}
    
\end{Example}

Interestingly, in \Cref{Example 1 xt 1}, after agents $1$–$3$ merge, they move away from agents $4$–$6$ despite no external force driving separation; paradoxically, agents $1$ and $4$ remain mutually attractive.

This behavior stems from inherent imprecision in implementing the sign function, which may occur both in numerical simulations and in physical systems.

In simulations, finite time steps cause sign-function imprecision near discontinuities, leading to chattering. As shown in \Cref{Example 1 xt 1} (right), agents can enter periodic collective motion. Such \textit{Steady-State Waveforms} have an average ``velocity''; for example, the waveform in \Cref{Example 1 xt 1} (right) propagates with velocity $-\tfrac{1}{3}$.

The system’s initial conditions dictate which Steady-State Waveform it converges to, producing distinct macroscopic behaviors. Even slight variations can switch outcomes, and the outcomes are inherently unpredictable. For example, with $x(0) = [0.001,1.0001,2,3.999,3,2]^T$ (\Cref{Example 1 xt 2}), at around $0.39$ s agents $1$–$3$ form a waveform with average velocity $1$ (\Cref{Example 1 xt 2}, middle), while agents $4$–$6$ form one with velocity $0$. By about $1.96$ s, the system reaches consensus, yielding a waveform with velocity $0$ (\Cref{Example 1 xt 2}, right).

Reducing the time step or using higher‑order solvers (e.g., Runge–Kutta) does not eliminate this phenomenon (see \Cref{ss:s:1}).

Due to non-uniqueness, saying that \Cref{e:c:3} ``achieves consensus'' can be misleading: the same initial condition may lead to both consensus and dissensus solutions, as shown in \Cref{Example 2} and \Cref{Example 1}.

In systems with non‑unique solutions, ``strong'' and ``weak'' qualifiers indicate whether a property holds for all solutions or merely for some. For example, the concepts of Stable and Weakly Stable in \cite[p. 152]{filippov2013differential}, the concept of weakly invariant set in \cite[Definition 2.7]{ryan1998integral}. 
We define Strong and Weak Consensus as follows.

\begin{Definition}[Strong Consensus / Weak Consensus]
The system of differential equations is said to achieve Strong Consensus (resp., Weak Consensus) if, for every initial condition, all solutions (resp., there exists at least one solution) $x(t)$ to the system are (resp., is a) consensus solution(s). 
\end{Definition}

As demonstrated in \Cref{Example 2} and \Cref{Example 1}, the system may still fail to reach Strong Consensus even without disturbance and with a strongly connected graph. This renders conventional consensus criteria (e.g., ``whether the communication graph is connected'') ineffective, while also invalidating traditional consensus rate estimates (e.g., ``the real part of the Laplacian matrix's second smallest eigenvalue''). 

\section{Consensus Analysis}\label{s:Consensus Analysis}
We study consensus by analyzing the evolution of the difference between the maximum and minimum agent states. For a system state \(x \in \mathbb{R}^n\), define:
\begin{itemize}
\item $M(x) := \max_{i\in \mathcal{V}} x_i$, the maximum agent state;
\item $m(x) := \min_{i\in \mathcal{V}} x_i$, the minimum agent state;
\item $\D{x} := M(x) - m(x) = \max_{i,j\in \mathcal{V}}(x_i-x_j)$, the maximal difference;
\item \(S_M(x) := \{i\in \mathcal{V}\mid x_i = M(x)\}\), the set of agents with maximum agent state;
\item \(S_m(x) := \{i\in \mathcal{V}\mid x_i = m(x)\}\), the set of agents with minimum agent state.
\end{itemize}

The quantity $\D{\cdot}$ is a seminorm on $\mathbb{R}^n$, referred to as the range seminorm. It equal to twice the seminorm $\left|\!\left|\!\left|\cdot\right|\!\right|\!\right|_{dist,\infty}$ defined in \cite{de2023dual}.

\begin{Lemma}\label{t:m:2}
    Let $x:\mathbb{R}\to\mathbb{R}^n$ be a solution of \Cref{e:c:3} where $\mathcal{W}\in \mathbb{R}_{\ge 0}^{n \times n}$ and $d(\cdot,\cdot)$ satisfies \Cref{d satisfies}. Then the derivatives $\frac{\mathrm{d}}{\mathrm{d}t}M(x(t))$ and $\frac{\mathrm{d}}{\mathrm{d}t}m(x(t))$ exist a.e. on $t\in \mathbb{R}$, and satisfy
    \begin{equation}\label{e:m:1}
        \begin{aligned}
            & \frac{\mathrm{d}}{\mathrm{d}t}M(x(t)) \in \left(\operatorname{span}(\pmb{1})\cap\mathcal{F}_{S_M(x(t))}(t, x(t))\right)_1,\\
            & \frac{\mathrm{d}}{\mathrm{d}t}m(x(t)) \in \left(\operatorname{span}(\pmb{1})\cap\mathcal{F}_{S_m(x(t))}(t, x(t))\right)_1.
        \end{aligned}
    \end{equation}
\end{Lemma}
\begin{proof}
See \Cref{Proof of t:m:2}.
\end{proof}

For each nonempty $V\subset\mathcal{V}$, define $\alpha^V,\beta^V\in \mathbb{R}^{\operatorname{card}(V)}$ as
\begin{equation}\label{def of alpha beta}
        \alpha^V_i = -\sum_{j\in \mathcal{V}} w_{ji} \quad \forall i\in V, \qquad
    \beta^V_i = \sum_{j\in V} w_{ji} -\sum_{j\in \mathcal{V}\backslash V} w_{ji} \quad \forall i\in V.
\end{equation}  

The $\operatorname{card}(V)$-dimensional hyperrectangle
\begin{equation}
    S(V):=\bigl\{y\in\mathbb{R}^{\operatorname{card}(V)}\mid \alpha^V\preceq y\preceq\beta^V\bigr\}
\end{equation}
in certain cases contains $\mathcal F_V(t,x)$.

\begin{Lemma}\label{c f F subset S}
    Let $\mathcal{W}\in \mathbb{R}_{\ge 0}^{n\times n}$ and $d(\cdot,\cdot)$ satisfies \Cref{d satisfies}. Then for every $(t,x) \in \mathbb{R}^{n+1}$,
    \begin{equation}
            \mathcal{F}_{S_M(x)}(t,x) \subset S(S_M(x)), \qquad
            -\mathcal{F}_{S_m(x)}(t,x) \subset S(S_m(x)).
    \end{equation}
\end{Lemma}
\begin{proof}
    See \Cref{Proof of c f F subset S}.
\end{proof}

Define the Autonomy Index $\Q: \mathcal{P}(\mathcal{V})\setminus\{\emptyset\}\to \overline{\mathbb{R}}$ as follows:
\begin{equation}
    \Q(V) := \sup\left( \operatorname{span}(\pmb{1}) \cap S(V)\right)_1,
\end{equation}
where $\overline{\mathbb{R}} := \{-\infty,\infty\}\cup \mathbb{R}$, and $\sup$ represents the supremum. We formally write $\sup \emptyset = -\infty$ to account for the case where $\operatorname{span}(\pmb{1}) \cap S(V) = \emptyset$. We also refer to the index $\Qu$ defined in \Cref{eq in t:m:5} as the Autonomy Index, although its value may differ in certain cases; see \Cref{s:Computation of A} for details.

Then by \Cref{t:m:2} and \Cref{c f F subset S}, for almost all $t$, 
\begin{equation}
    -\Q(S_m(x(t))) \le \frac{\mathrm{d}}{\mathrm{d}t}m(x(t)), \qquad \frac{\mathrm{d}}{\mathrm{d}t}M(x(t)) \le \Q(S_M(x(t))).
\end{equation}

Here $\Q(V)$ is the maximal rate at which agents in $V\subset\mathcal V$ can jointly move in the positive (or negative) direction when they form the maximal (or minimal) set, and $\Q(V_1) + \Q(V_2)$ is the maximal rate at which two group of agents $V_1\subset \mathcal{V}$ and $V_2\subset \mathcal{V}$, serving as the maximal and minimal sets respectively, move away from each other.

Since any disjoint nonempty $V_1, V_2\subset\mathcal V$ can arise as the maximal and minimal sets, respectively, we define the Polarization Index $\A:\mathbb{R}_{\ge 0}^{n\times n} \to \mathbb{R}$ as follows:
\begin{equation}\label{def of A}
    \A(\mathcal{W}) := \max_{\substack{V_1,V_2\subset\mathcal{V}\\V_1,V_2\neq\emptyset,\ V_1\cap V_2=\emptyset}}[\Q(V_1)+\Q(V_2)].
\end{equation}

\Cref{eq max min op problem} is an equivalent definition of $\A(\mathcal{W})$, and its equivalence is established in \Cref{t:m:5}.

\begin{Theorem}\label{t:m:3}
    Let $\mathcal{W}\in \mathbb{R}_{\ge 0}^{n\times n}$. Then for every solution $x:\mathbb{R}\to \mathbb{R}^n$ of \Cref{e:c:3} with $d(\cdot,\cdot)$ satisfies \Cref{d satisfies} and $t^0\in\mathbb{R}$,
    \begin{equation}\label{eq in t:m:3}
    \D{x(t)} \le \max\Big\{0,\D{x(t^0)}+\A(\mathcal{W})t\Big\} \quad \forall t \ge t^0.
    \end{equation}

    Moreover, for every $r\in \mathbb{R}_{\ge 0}$, there exists a solution $x'(\cdot)$ of \Cref{e:c:3} with $d'(\cdot,\cdot)$ satisfying \Cref{d satisfies} and $\D{x'(0)} = r$ such that
    \begin{equation}\label{eq 2 in t:m:3}
    \D{x'(t)} = \max\Big\{0,\D{x'(0)}+\A(\mathcal{W})t\Big\} \quad \forall t \ge 0.
    \end{equation}
\end{Theorem}
\begin{proof}
See \Cref{Proof of t:m:3}.    
\end{proof}

Define $T:\mathbb{R}_{\ge 0}\to \mathbb{R}_{\ge 0}$ as
\begin{equation}\label{eq in t:m:4}
    T(r) = \frac{r}{|\A(\mathcal{W})|}.
\end{equation}

\begin{Corollary}\label{t:m:4}
    Let $\mathcal{W}\in \mathbb{R}_{\ge 0}^{n\times n}$. Then, a necessary and sufficient condition for \Cref{e:c:3} with any $d(\cdot,\cdot)$ satisfies \Cref{d satisfies} to achieve Strong Consensus is:
    \begin{equation}
        \A(\mathcal{W}) < 0
    \end{equation}

    If $\A(\mathcal{W}) < 0$, all solutions $x(\cdot)$ of \Cref{e:c:3} with some $d(\cdot,\cdot)$ satisfying \Cref{d satisfies} achieve consensus no later than $T(\D{x(0)})$, i.e., $x_i = x_j$ for all $t \ge T(\D{x(0)})$ and all $i, j \in \mathcal{V}$. Moreover, for every $r \in \mathbb{R}_{\ge 0}$, $T(r)$ is the least upper bound on the consensus time over all solutions of \Cref{e:c:3} with $d(\cdot,\cdot)$ satisfying \Cref{d satisfies} and $\D{x(0)} = r$.
\end{Corollary}
\begin{proof}
    This is a direct corollary of \Cref{t:m:3}.
\end{proof}

\section{Computation of $\A(\mathcal{W})$}\label{s:Computation of A}

In this section, we present an algorithm for computing $\A(\mathcal{W})$. Although its worst-case complexity is exponential, pruning strategies significantly reduce average-case runtime in practice.  

The following theorem transforms the computation of $\A(\mathcal{W})$ into an integer optimization problem with a max–min structure.
\begin{Theorem}\label{t:m:5}
    For $\mathcal{W}\in \mathbb{R}_{\geq 0}^{n \times n}$, define
    \begin{equation}\label{eq in t:m:5}
        \Qu(V) := \min_{i\in V}[\sum_{j\in V} w_{ji} -\sum_{j\in \mathcal{V}\backslash V} w_{ji}],
    \end{equation}
    then
    \begin{equation}\label{eq max min op problem}
    \begin{aligned}
        \A(\mathcal{W}) = \max_{\substack{V_1,V_2\subset\mathcal{V}\\V_1,V_2\neq\emptyset,\ V_1\cap V_2=\emptyset}}[\Qu(V_1)+\Qu(V_2)].
    \end{aligned}
    \end{equation}
\end{Theorem}
\begin{proof}
    See \Cref{Proof of t:m:5}. 
\end{proof}

By \Cref{l f Q inf iif Ql Qu}, \( \Q(V) = \Qu(V) \) if and only if $\Q(V) \not = -\infty$. The only difference between the two is that \( \Q(V) \) can determine whether \( V \) can consistently remain as the maximal set (or minimal set) during a time interval \([a, b]\) (where \( a < b \)). $S(V)$ represents the set of possible velocity vectors of agents in $V$ when they form the maximal set. The velocity vectors in $\operatorname{span}(\pmb{1})$ indicate that all agents in $V$ have identical velocities. If $\Q(V) = -\infty$ then the intersection $\operatorname{span}(\pmb{1}) \cap S(V)$ is empty , they cannot move at the same velocity and thus cannot maintain identical states, causing the disintegration of subgroup $V$. The definition of $\Q(V)$ is geometric and offers intuitive insight, while $\Qu(V)$ is arithmetic and provides stronger interpretability.

The following lemma reduces the maximization domain.

\begin{Lemma}\label{t:m:6}
    Let $\mathcal{W} \in \mathbb{R}_{\geq 0}^{n \times n}$ and define $\Qu(V)$ as in \Cref{eq in t:m:5}.  
    Set
    \begin{equation}\label{eq 0 in t:m:6}
       z^* := \max_{\substack{V\subset\mathcal{V}\\V\neq\emptyset,\ V\not=\mathcal{V}}}[\Qu(V)+\Qu(V^c)].
    \end{equation}
    
    Then $\A(\mathcal{W}) \ge z^*$, and equality holds if $z^* \le 0$.
\end{Lemma}
\begin{proof}
    See \Cref{Proof of t:m:6}. 
\end{proof}

We analyze the optimization problem given in \Cref{eq 0 in t:m:6}, which can be equivalently written as
\begin{equation}\label{eq int op problem}
\begin{aligned}
    \max_{a \in \{-1,1\}^n} \ \min_{b \in \{-1,0,1\}^n} \quad & a^T \mathcal{W} b \\
    \text{subject to} \quad & -n < \pmb{1}^T a < n, \\
                            &\pmb{1}^T b = 0, \\
                            & a^T b = 2, \\
                            & a_i b_i \ge 0, \quad \forall i \in \mathcal{V}.
\end{aligned}
\end{equation}

Here, the vector $a$ indicates whether each node in $\mathcal{V}$ belongs to $V$ or $V^c$. The condition $-n < \pmb{1}^T a < n$ corresponds to the requirement that $V \neq \emptyset$ and $V \neq \mathcal{V}$. 

Any vector $b$ satisfying the constraints must have exactly one entry equal to $1$ and one equal to $-1$, with all others being zero. Moreover, the positive and negative entries of $b$ must lie in the coordinates where $a$ takes values $1$ and $-1$, respectively.

\begin{Lemma}\label{t:m:7}
    The optimization problem \Cref{eq int op problem} and
    \begin{equation}\label{eq ld dual op problem}
        \begin{aligned}
            \max_{a \in \{-1,1\}^n,c\in \mathbb{R}^{n+2}} \quad & 2 c_2\\
            \text{subject to} \quad & -n < \pmb{1}^T a < n, \\
                            &\begin{bmatrix}
                                \pmb{1} & a & \operatorname{diag}(a)
                            \end{bmatrix} c = \mathcal{W}^Ta,\\
                            & c_i \ge 0, \quad \forall i\in \{3,4,\dots,n+2\},
        \end{aligned}
    \end{equation}
    attain the same optimal value.
\end{Lemma}
\begin{proof}
    See \Cref{Proof of t:m:7}. 
\end{proof}

\begin{Theorem}\label{t:m:8}
    Let $z^\dagger$ denote the optimal value of the optimization problem shown in \Cref{eq ld dual op problem}. Then $ \A(\mathcal{W}) \ge z^\dagger$, and equality holds if $z^\dagger \le 0$.
\end{Theorem}
\begin{proof}
    This is a direct corollary of \Cref{t:m:5}, \Cref{t:m:6}, and \Cref{t:m:7}.
\end{proof}

\Cref{t:m:8} shows that $\A(\mathcal{W})$ can be computed using \Cref{eq ld dual op problem} only if $z^\dagger \leq 0$. However, this condition does not affect consensus determination or consensus time estimation, as we can always determine the sign of $\A(\mathcal{W})$ based on the sign of $z^\dagger$, and the equality holds when $z^\dagger \leq 0$.

Thus, we have simplified the computation of $\A(\mathcal{W})$ to solving a mixed integer programming problem, provided that the optimal value of \Cref{eq ld dual op problem} is not greater than $0$. Experimental results show that by using the branch-and-bound method with parallel computing, problems of size $n = 45$ can be solved within 5 minutes on a personal computer. Detailed experimental data can be found in \Cref{table for run time}.

\begin{table}[htbp]
\centering
\caption{Runtime Comparison Between \Cref{eq max min op problem} and \Cref{eq ld dual op problem}}
\label{table for run time}
\begin{tabular}{c|c|c}
\toprule
\textbf{Problem Size ($n$)} & \textbf{\Cref{eq max min op problem} (s)}  & \textbf{\Cref{eq ld dual op problem} (s)} \\
\midrule
5  & 0.001 & 0.007 \\
10  & 0.022 & 0.019 \\
15  & 7.926 & 0.075 \\
20  & - & 0.318 \\
30  & - & 5.385 \\
40  & - & 113.732 \\
\bottomrule
\end{tabular}

\begin{tablenotes}
\small
\item \textit{All experiments were performed on an 11th Gen Intel\textsuperscript{\textregistered} Core\texttrademark{} i9-11900K @ 3.50GHz processor using a single CPU core. }
    
\item \textit{All values represent average runtimes over 20 independent runs.}

\item \textit{``-'' indicates that the computation did not finish within 30 minutes.}
    
\item \textit{The mixed integer program in \Cref{eq ld dual op problem} is implemented using \texttt{gurobipy}.}
\end{tablenotes}
\end{table}

\section{Modeling of Opinion Dynamics and Interpretation of Notation}\label{s:Modeling of Opinion Dynamics and Interpretation of Notation}
Each individual's opinion is a real number in $\mathbb{R}$. Social interactions are modeled by a weighted directed graph $\mathcal{G} = (\mathcal{V}, \mathcal{E}, \mathcal{W})$, where $\mathcal{V}$ is the set of individuals. A directed edge $(j, i) \in \mathcal{E}$ means individual $i$ can receive information from individual $j$. The weight $w_{ji}$ represents the degree of trust individual $i$ places in individual $j$, and $w_{ji} > 0$ if and only if $(j, i) \in \mathcal{E}$. The diagonal element $w_{ii}$ (i.e., a self-loop) reflects the maximal strength of individual $i$’s subjective judgment or external, unmodeled influences. 

We model opinion dynamics using \Cref{e:c:3}, referred to as the Comparative-Only Model. In this model, communication does not reveal exact opinions, and individuals may not have precise knowledge of their own. Instead, they can only make comparisons and update their opinions accordingly, moving toward those of individuals they trust more. 

This model characterizes an individual’s imprecise perception of their own opinion, which becomes gradually clearer through interactions and comparisons with others. And individuals may avoid fully expressing their views to prevent conflict or being seen as outliers. As a result, extreme opinions may not be fully revealed; communication may only allow individuals to infer who holds a more extreme or more moderate opinion. This resembles situations where individuals express merely ``approve'' or ``disapprove,'' offering little indication of the intensity behind their stance.

The term $d(\cdot,\cdot)$ in \Cref{e:c:3} reflects an individual’s subjective judgment or external, unmodeled influences. It may be a random function to capture the variability of an individual’s subjective judgment, such as shifts in attitude under different emotions, or as \Cref{d for fixed influence} to represent stably influenced by fixed information sources (e.g., books or media) or by persistent expectations.
\begin{equation}\label{d for fixed influence}
    d_i(t,x) = \sum_{j = 1}^{m}w'_{ji} \operatorname{sign}(s_j-x_i(t))
\end{equation}
where $s_j \in \mathbb{R}$ represents fixed external opinions or persistent expectations, $w'_{ji} \ge 0$ denotes their influence on individual $i$, and the total influence satisfies $\sum_{j=1}^{m} w'_{ji} \le w_{ii}$. 

A special case of \Cref{d for fixed influence} is \Cref{d for Ex 6},
\begin{equation}\label{d for Ex 6}
    d_i(t,x) = w_{ii} \operatorname{sign}(x_i(0)-x_i(t)),
\end{equation}
which indicates that each individual seeks to preserve their initial opinion, corresponding to the stubborn agents in 
\cite{tian2021social}, with $w_{ii}$ quantifying their degree of stubbornness. \Cref{fig Ex6} shows an experiment and illustrates: 1. groups with weaker internal connections tend to exhibit disagreement; 2. more stubborn groups (larger $w_{ii}$) are more likely to preserve their initial opinions.

\begin{figure}[htbp]
    \centering  
    \subfigure[Graph structure.]{\label{fig Ex6 G}\includegraphics[width=0.47\linewidth]{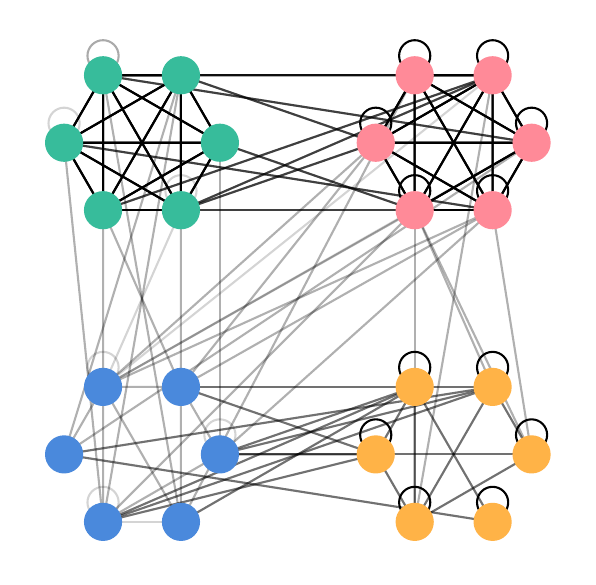}}
    \subfigure[Simulation result.]{\label{fig Ex6 xt}\includegraphics[width=0.48\linewidth]{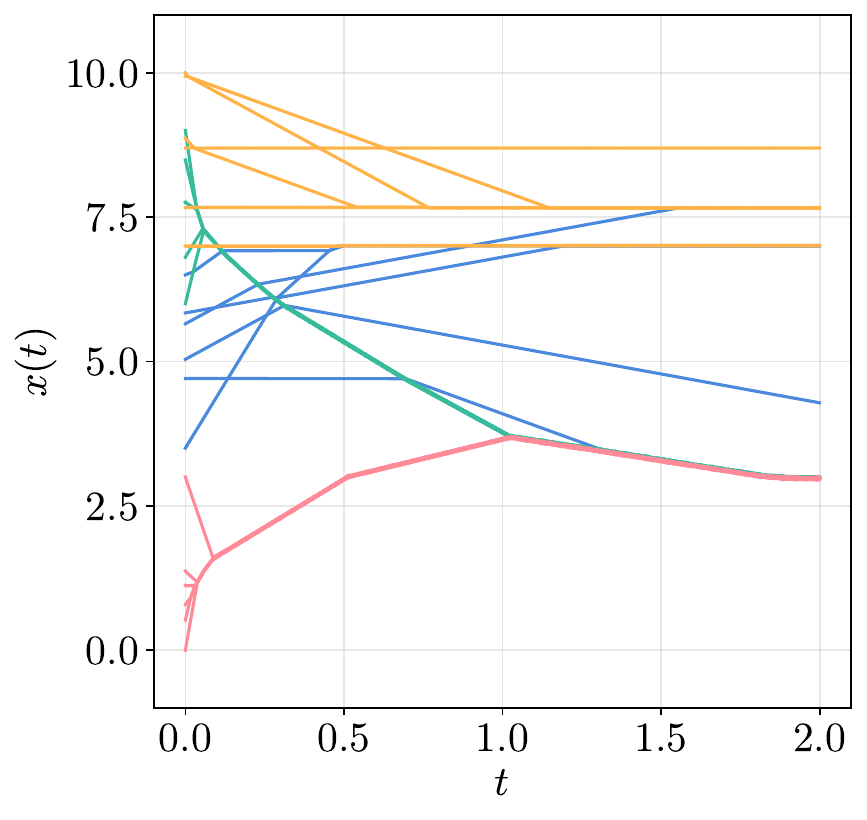}}
    \caption{Graph structure and simulation results of system \Cref{e:c:3} with $d(\cdot,\cdot)$ from \Cref{d for Ex 6}. In \Cref{fig Ex6 G}, darker edges indicate larger weights. The figure shows four communities: the pink and yellow communities are more stubborn (larger $w_{ii}$), while the green and pink communities have strong internal connectivity. The blue community had initial opinions in $[3.5, 6.5]$, which by $t = 2$ s spread to $[3, 7.6]$. The green community had initial opinions in $[6, 9]$, which by $t = 2$ s converged to $3$. The pink community had initial opinions in $[0, 3]$, which by $t = 2$ s converged to $3$. The yellow community had initial opinions in $[7, 10]$, which by $t = 2$ s narrowed to $[7, 8.7]$. Only the green and pink communities achieved internal consensus, and only the individuals in the pink and yellow communities kept their opinions within their initial opinion intervals.}
    \label{fig Ex6}
\end{figure}

\subsection{$\Qu$ and $\A$ Revisited: Community-Based Conditions for Consensus}

By \Cref{t:m:4}, dissensus arises when $\A \ge 0$. We now explain the meaning of $\Qu(V)$ and $\A$ in the context of opinion dynamics, and show that the absence of Strong Communities \cite{radicchi2004defining} is a sufficient condition for Strong Consensus, while the absence of Satisfactory Partitions \cite{bazgan2006satisfactory} is a necessary condition.

In the definition of $\Qu(V)$, 
$\sum_{j\in V} w_{ji} -\sum_{j\in \mathcal{V}\backslash V} w_{ji}$ represents the difference between two basic vertex community variables: the Internal Strength and External Strength \cite{fortunato2016community}. The Internal Strength of node $i$ is $\sum_{j \in V} w_{ji}$, the total weighted in-degree from nodes in $V$, while the External Strength is $\sum_{j \in \mathcal{V} \setminus V} w_{ji}$, the total weighted in-degree from nodes outside $V$.

An LS-set \cite{luccio1969decomposition}, or Strong Community \cite{radicchi2004defining}, is defined as a subgraph where every node of this subgraph has a higher Internal than External Strength.
\begin{Definition}[Strong Community \cite{radicchi2004defining}]
    Let $\mathcal{G} = (\mathcal{V}, \mathcal{E}, \mathcal{W})$. A Strong Community is a subset $V \subset \mathcal{V}$ such that every node in $V$ has higher Internal than External Strength, i.e.,
    \begin{equation}
    \sum_{j \in V} w_{ji} > \sum_{j \in \mathcal{V} \setminus V} w_{ji}, \qquad \forall i \in V.
    \end{equation}
\end{Definition}

Therefore, $\Qu(V) > 0$ holds if and only if $V$ is a Strong Community. Compared to the binary classification of whether a set $V$ is a Strong Community, $\Qu(V)$ provides a continuous measure of community strength. \Cref{fig for evolution of a weighted} illustrates this well: as $\Qu(V_1)$ and $\Qu(V_2)$ increase, both $V_1$ and $V_2$ tend to exhibit stronger internal connectivity and weaker external connectivity. Light blue and light red nodes highlight the key nodes within $V_1$ and $V_2$, respectively, i.e., the ones belonging to $\arg\min_{i\in V}[\sum_{j\in V} w_{ji} -\sum_{j\in \mathcal{V}\backslash V} w_{ji}]$ where $\arg\min_{x}f(x) = \{x\mid f(x) = \min_{x'} f(x')\}$. For any group, its key nodes are those that critically influence the group’s cohesion and capacity for independent judgment. These nodes are the most susceptible within the group to external influence, and thus have the greatest potential to separate from the group or affect the group’s independence. Strengthening internal connections to key nodes while weakening external influences on them enhances the group's cohesion and capacity for independent judgment. This process is illustrated step by step in \Cref{fig for evolution of a weighted}.

\begin{figure}[htbp]
\centerline{\includegraphics[width=0.9\linewidth]{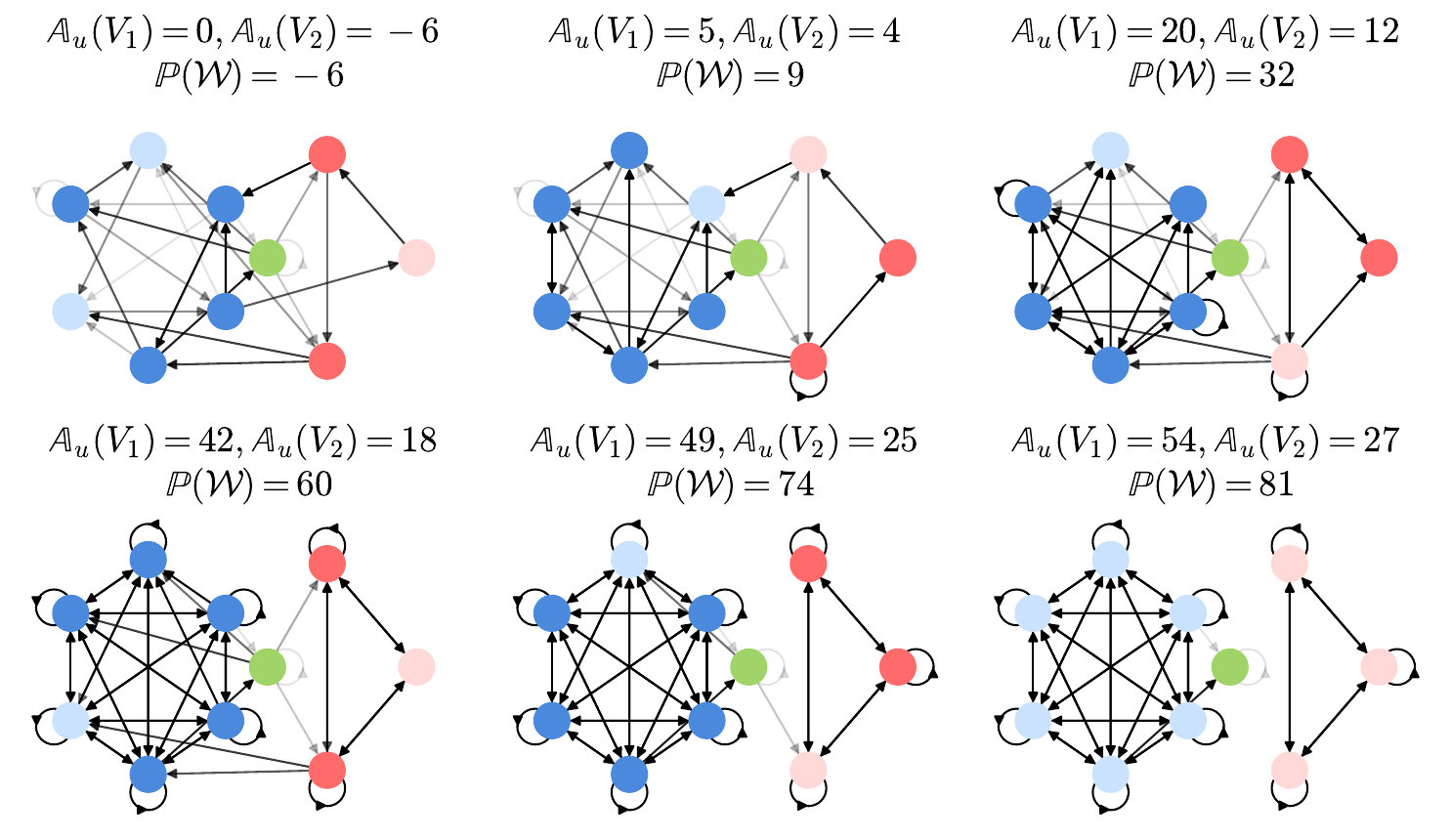}}
    \caption{Evolution of a weighted directed graph with $n = 10$ nodes and edge weights in $\{0, 1, \dots, 9\}$, showing a sequence of modifications that increase the value of $\A(\mathcal{W})$. Edge darkness indicates weight magnitude (Darker edges = larger weights). Each graph illustrates a pair of disjoint subsets $(V_1, V_2)$ that maximize $\Qu(V_1) + \Qu(V_2)$. Blue and red nodes represent the elements of $V_1$ and $V_2$, respectively, while green nodes belong to neither set. Light blue and light red nodes highlight the key nodes within $V_1$ and $V_2$, respectively, i.e., the ones that attain the minimum in the definition of $\Qu(\cdot)$. As the internal connectivity within each of the subgraphs $V_1$ and $V_2$ becomes stronger, and their connections to the rest of the network become weaker, both $\Qu(V_1)$, $\Qu(V_2)$, and $\A(\mathcal{W})$ increase.}
    \label{fig for evolution of a weighted}
\end{figure}

\begin{Corollary}\label{t:sufficient condition}
    A sufficient condition for system \Cref{e:c:3} to achieve Strong Consensus is that the graph $\mathcal{G}$ contains no Strong Communities.
\end{Corollary}
\begin{proof}
    If the graph $\mathcal{G}$ contains no Strong Communities, then $\Qu(V) < 0$ for all nonempty $V\subset \mathcal{V}$. The conclusion follows directly from \Cref{t:m:4} and \Cref{t:m:5}.
\end{proof}

Another related concept is the Satisfactory Partition Problem \cite{bazgan2006satisfactory}, which asks whether a graph can be divided into two disjoint nonempty subsets $V_1$ and $V_2$ such that both are Strong Communities. 

\begin{Corollary}\label{t:necessary condition}
     A necessary condition for system \Cref{e:c:3} to achieve Strong Consensus is that the graph $\mathcal{G}$ has no Satisfactory Partition.
\end{Corollary}
\begin{proof}
    Suppose the graph $\mathcal{G}$ admits a Satisfactory Partition $(V_1,V_2)$.
    
    Then $\Qu(V_1) > 0$, $\Qu(V_2) > 0$ and by \Cref{t:m:5}, we have 
    \begin{equation}
        \A \geq \Qu(V_1) + \Qu(V_2) > 0.
    \end{equation}
    
    Thus, by \Cref{t:m:4}, Strong Consensus is unachievable.
\end{proof}

$\Qu(V_1)+\Qu(V_2)$ serves as a modularity-like measure, analogous to that introduced in \cite{newman2006modularity}, for evaluating the strength of a partition $(V_1, V_2)$ of the graph, where higher values indicate denser intra-community and sparser inter-community connections. If a pair $(V_1, V_2)$ satisfies $\Qu(V_1) + \Qu(V_2) = \A$, then the graph’s optimal division is $(V_1, V_2, V_3)$ if $V_3 = (V_1 \cup V_2)^c$ is nonempty, and $(V_1, V_2)$ otherwise. Here, $\A$ represents the supremum of such partition scores for a given graph: it tends to be large when the graph exhibits two well-separated communities, and small otherwise. \Cref{fig for evolution of a weighted} and \Cref{fig for Several unweighted undirected graphs} illustrate this for weighted directed and unweighted undirected graphs, respectively.

\begin{figure}[htbp]
    \centerline{\includegraphics[width=1\linewidth]{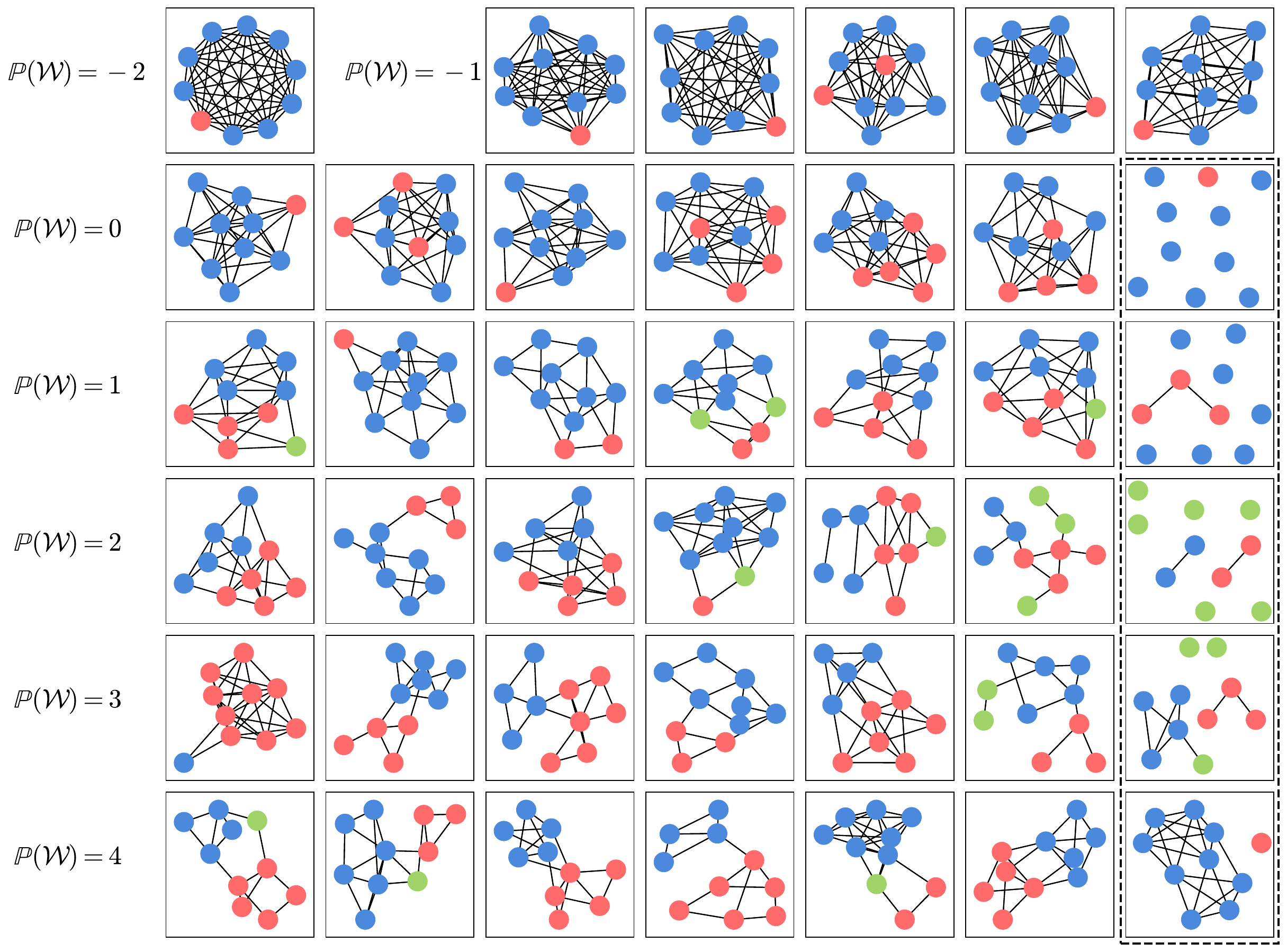}}
    \caption{Several unweighted, undirected graphs $\mathcal{W}\in \{0,1\}^{10\times 10}$ with $n = 10$ nodes, all without self-loops, each corresponding to a randomly generated graph with $\A(\mathcal{W})$ ranging from $-2$ to $4$. Each subgraph illustrates a pair of sets $(V_1, V_2)$ that maximizes $\Qu(V_1) + \Qu(V_2)$. Nodes in $V_1$ are colored blue, nodes in $V_2$ red, and nodes in $V_3 = (V_1 \cup V_2)^c$ green. The figures enclosed in the dashed box on the right are disconnected graphs, while the others are connected graphs. As $\A(\mathcal{W})$ increases, the connectivity within each of the subgraphs $V_1$ and $V_2$ becomes denser, while the connections between these subgraphs and the rest of the network become sparser, indicating a more pronounced community structure. }
    \label{fig for Several unweighted undirected graphs}
\end{figure}

Between the sufficient condition in \Cref{t:sufficient condition} and necessary condition in \Cref{t:necessary condition} lies the case where a Strong Community $V$ exists but no Satisfactory Partition can be found. In such cases, Strong Consensus may still be possible. An example is a music band and its fanbase: suppose the band members form a tightly connected Strong Community and their connections to fans are mostly one-directional (from band to fans). Even if the band’s opinions are only weakly influenced by the fanbase, their opinions can quickly spread throughout the entire fanbase, allowing the band and its fanbase as a whole to reach consensus. What hinders consensus is the presence of two internally cohesive but externally loosely connected subgroups, such as opposing political parties.

\subsection{Unpredictable Collective Dynamics}

The system exhibits a striking phenomenon: as shown in \Cref{Example 1 xt 1}, the opinions of agents $1$–$3$ evolve in a coordinated pattern and collectively drift toward increasing extremity in the negative direction, even though none of the agents initially held such extreme opinions (not shown in the figure, but their opinions would continue to decrease at a rate of $-\frac{1}{3}$, eventually falling below $0$). Such behavior reflects a form of collective confidence: external affirmation strengthens individual confidence, while reciprocal affirmation within the group endows the group itself with the capacity to act spontaneously against external influences. 

Moreover, as demonstrated in \Cref{Example 1}, small variations in the behavior of any individual can affect the collective trajectory, and the resulting effect is unpredictable. This phenomenon arises from the non-uniqueness of solutions in continuous-time systems. Each numerical simulation necessarily approximates some solution within the solution set of the continuous system; however, the solution is not unique, the particular trajectory to which a given simulation converges cannot be predetermined. 

\section{Illustrative Examples}\label{s:Illustrative Examples}

\subsection{Extension for \Cref{Example 1}}\label{ss:s:1}

\begin{figure}[htbp]
    \centering  
    \subfigure[Graph structure]{\raisebox{0.05\height}{\label{fig Example 3.1}\includegraphics[width=0.36\linewidth]{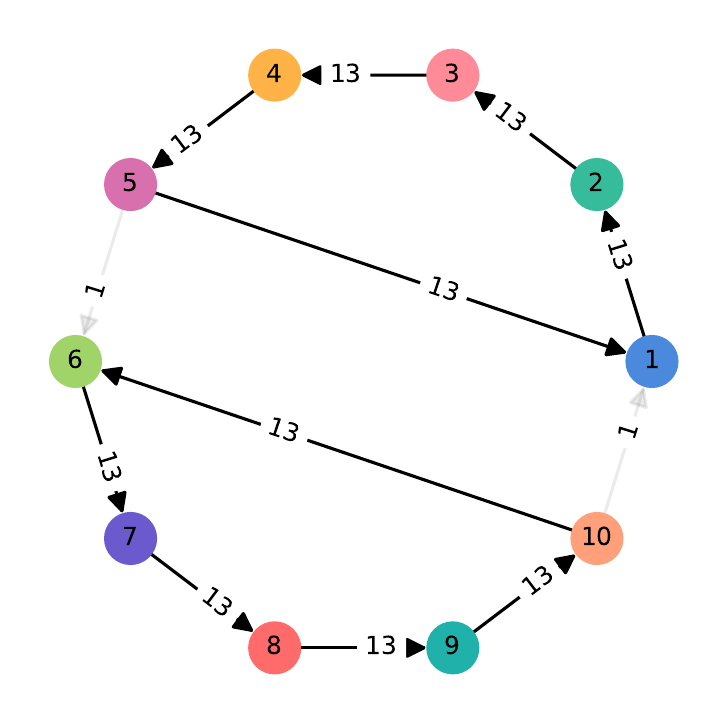}}}
    \hspace{-0.016\linewidth}
    \subfigure[{Simple Euler method, time steps $2^{-10}$ s; achieves consensus.}]{\label{fig Example 3.2}\includegraphics[width=0.36\linewidth]{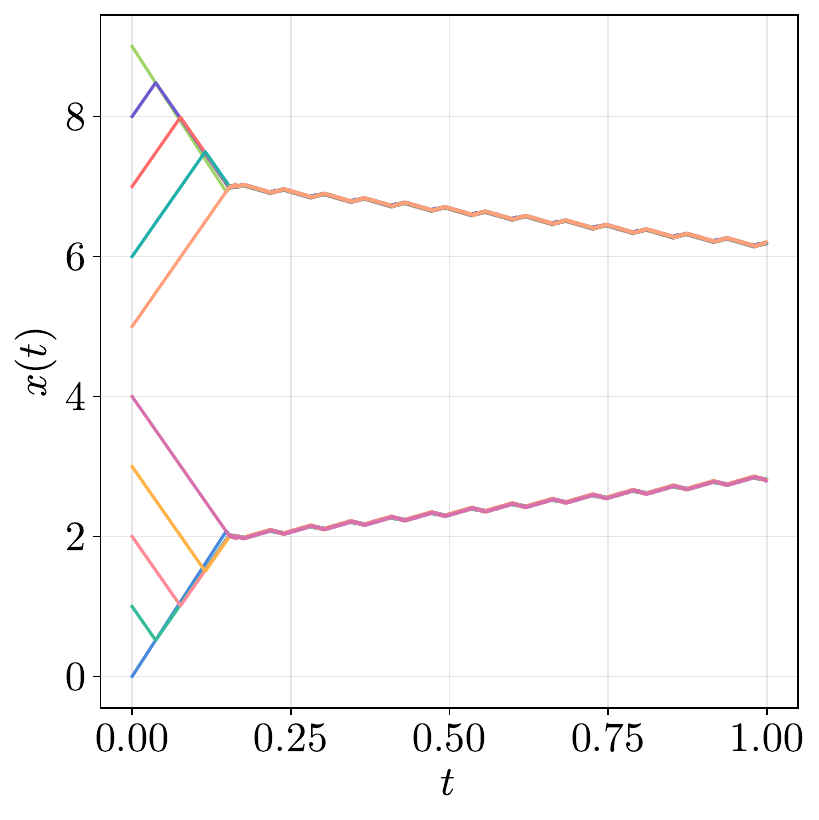}}
    \subfigure[{Runge–Kutta 4 method, time steps $2^{-10}$ s; exhibits dissensus.}]{\label{fig Example 3.3}\includegraphics[width=0.36\linewidth]{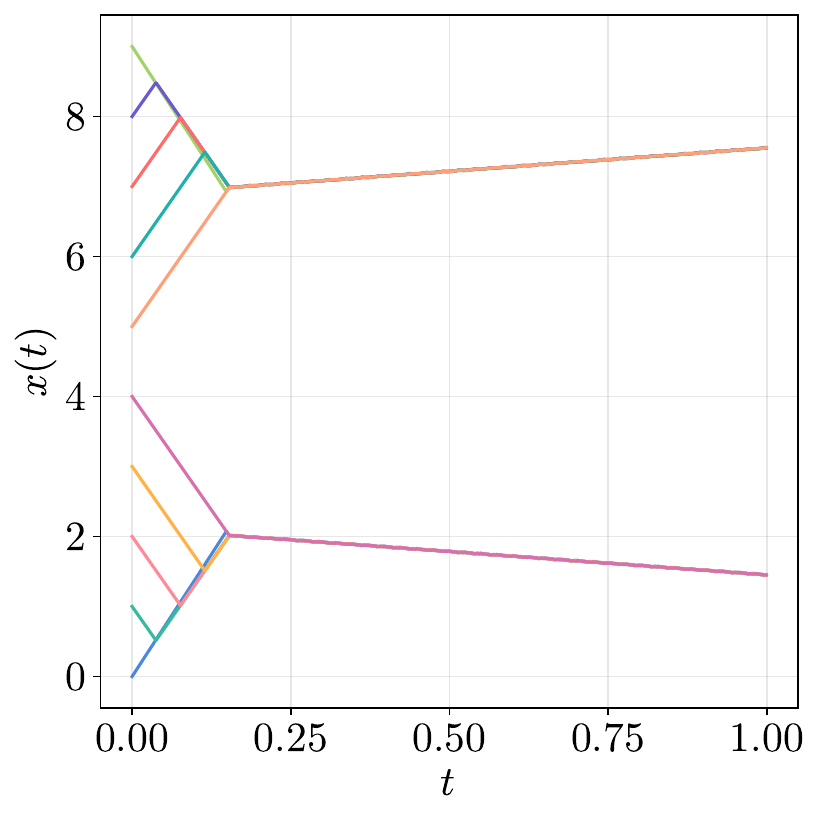}}
    \subfigure[{Runge–Kutta 4 method, time steps $2^{-20}$ s; exhibits dissensus.}]{\label{fig Example 3.4}\includegraphics[width=0.36\linewidth]{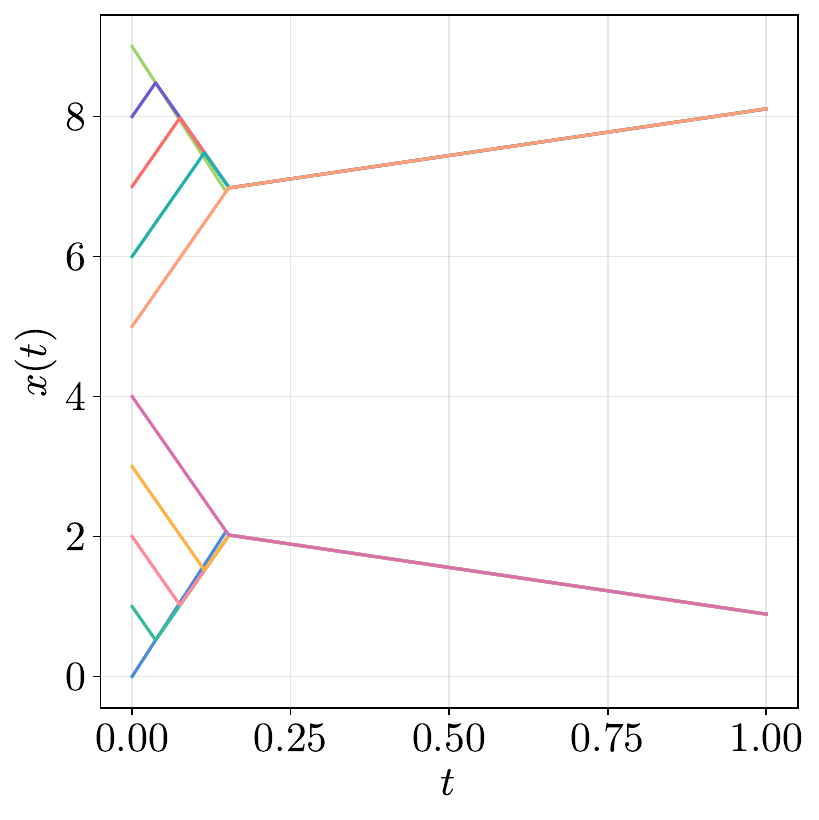}}
    \caption{Graph structure and simulation results under different conditions in \Cref{ss:s:1}, with initial value $x(0) = [0,1,2,3,4,9,8,7,6,5]^T$. Experiments show that when $\A > 0$, whether the simulation reaches consensus depends on the choice of parameters. Using higher-order solvers or smaller time steps does not guarantee consensus.}
    \label{fig Example 3}
\end{figure}

This example shows that even without disturbances and with a strongly connected graph, $\A(\mathcal{W})$ can be positive, allowing dissensus solutions. Simulations confirm their existence, and neither smaller step sizes nor higher-order solvers (e.g., RK4) can eliminate them.

Consider the case $n = 10$, with a strongly connected graph $\mathcal{W}$ illustrated in \Cref{fig Example 3.1} and let $d(t,x) \equiv \pmb{0}$. Here, $\A(\mathcal{W}) = 24 > 0$, indicating possible dissensus solution. The simulation results are shown in \Cref{fig Example 3} (see caption for details).

The above phenomenon also appears in unweighted directed graphs. In fact, the weighted graph from the previous example can be transformed into an equivalent unweighted one by the following procedure: For each node $i$, let $p$ be the maximum weight among its outgoing edges. If $p > 1$, split $i$ into $p$ subnodes $i^1, i^2, \dots, i^p$, redirect incoming edges to all subnodes, and replace each edge $(i,j)$ of weight $k$ with $k$ unit-weight edges $(i^1,j), \dots, (i^k,j)$. \Cref{Example 3 node spilt} illustrates an example of such a construction.

\begin{figure}[htbp]
    \centerline{\includegraphics[width=0.3\columnwidth]{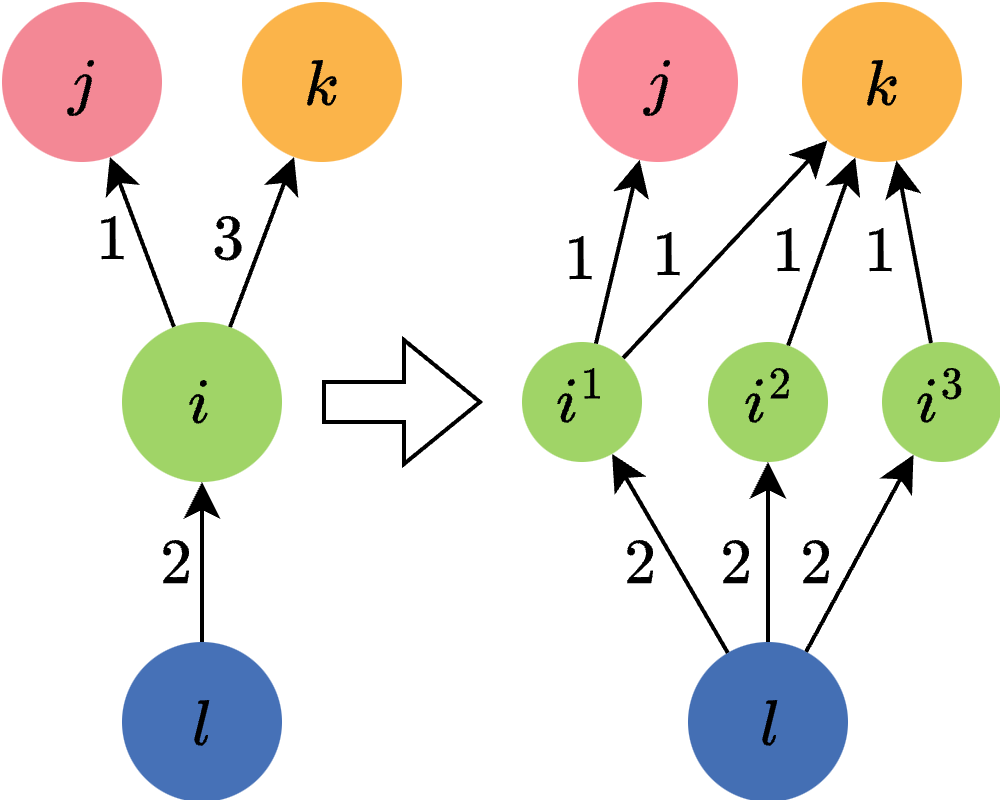}}
    \caption{Illustration of the node-splitting process for node $i$: since its maximum out-edge weight is 3, it is split into subnodes $i^1, i^2, i^3$. The incoming edge $(l, i)$ becomes $(l, i^1), (l, i^2), (l, i^3)$; the out-edge $(i, j)$ becomes $(i^1, j)$; and the weight-3 edge $(i, k)$ becomes $(i^1, k), (i^2, k), (i^3, k)$. This converts all outgoing weights to $1$.}
    \label{Example 3 node spilt}
\end{figure}

By setting all split subnodes to share the original node’s initial condition, i.e., $x_{i^1}(0) = \dots = x_{i^p}(0) = x_i(0)$, there exists a solution where these subnodes remain identical over time and follow the same trajectory as the original node in the weighted system, i.e., $x_{i^1}(t) = \dots = x_{i^p}(t) = x_i(t)$ for all $t$.

\subsection{Example of $\A(\mathcal{W}) > 0$ with Worst Disturbance
}\label{ss:s:2}

\begin{figure}[htbp]
    \centering  
    \hspace{-0.012\linewidth}
    \subfigure[Graph structure. Darker edge colors indicate higher edge weights.]{\label{fig ss:s:2 1}\includegraphics[width=0.47\linewidth]{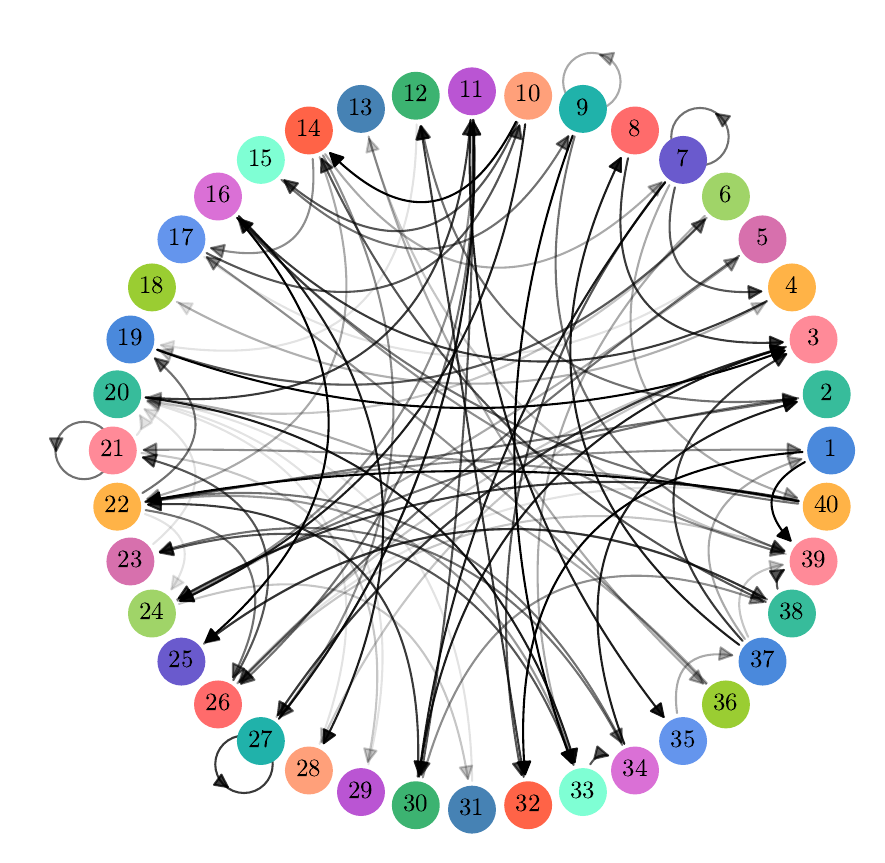}}
    \hspace{0.006\linewidth}
    \subfigure[Simulation from initial condition leading to consensus.]{\label{fig ss:s:2 2}\includegraphics[width=0.48\linewidth]{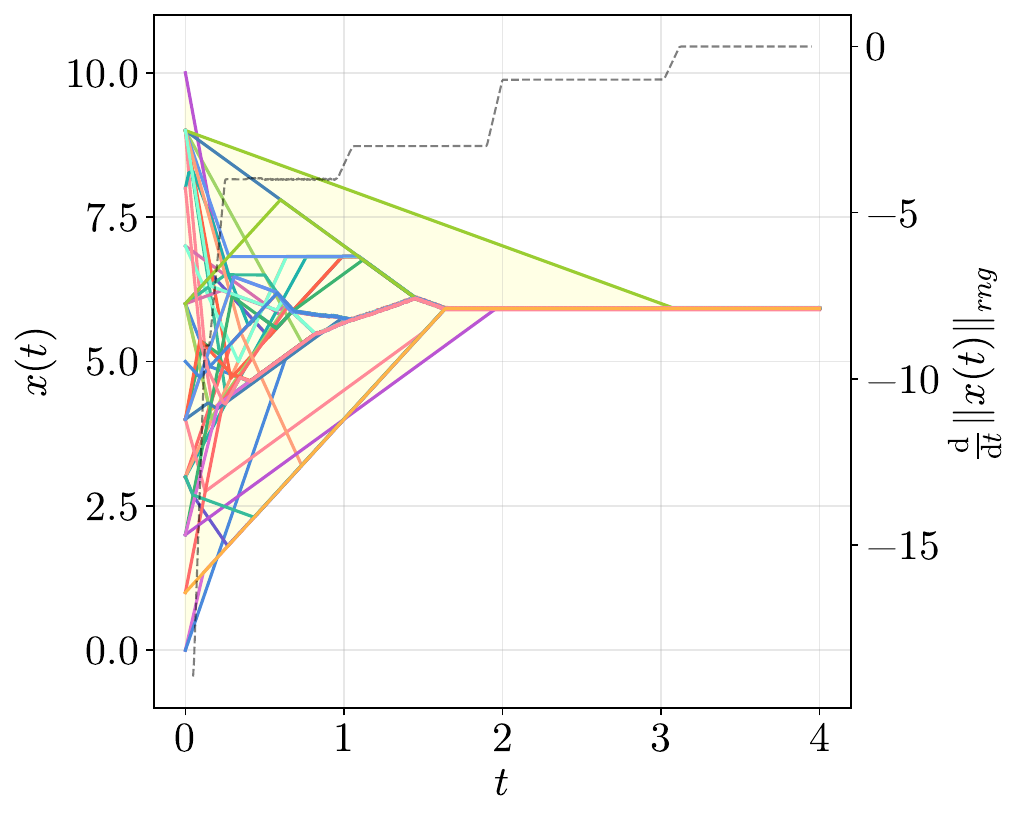}}
    \subfigure[Simulation from two initial conditions leading to dissensus.]{\label{fig ss:s:2 3}\includegraphics[width=0.47\linewidth]{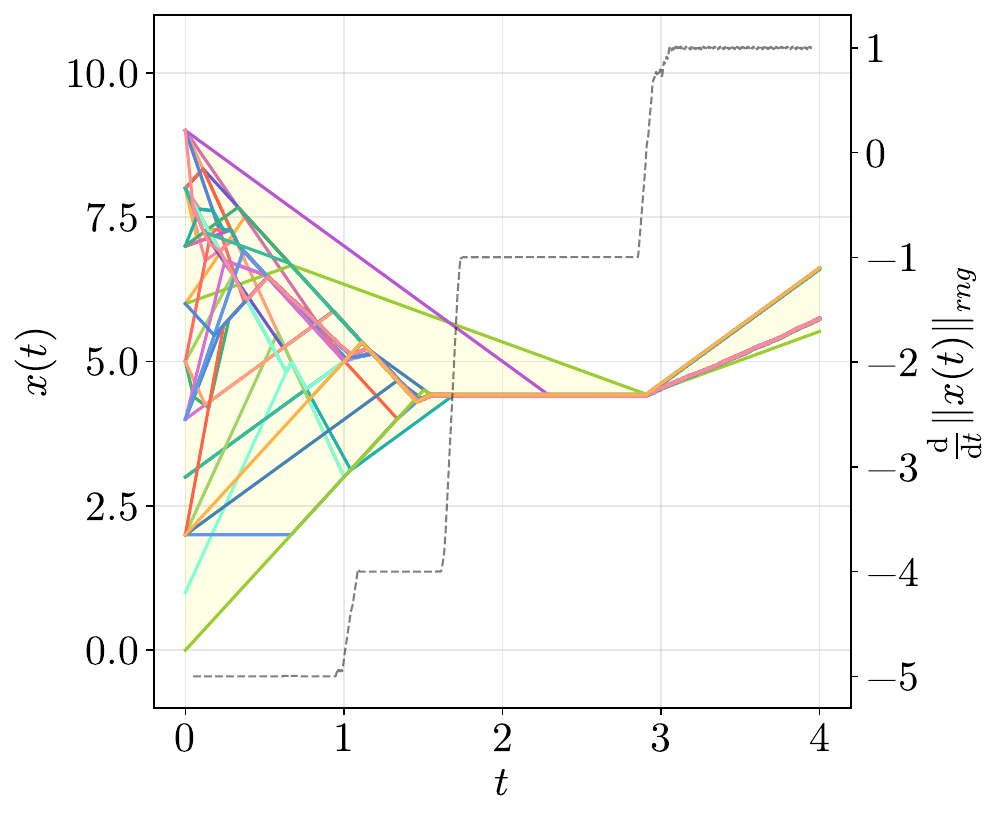}\includegraphics[width=0.47\linewidth]{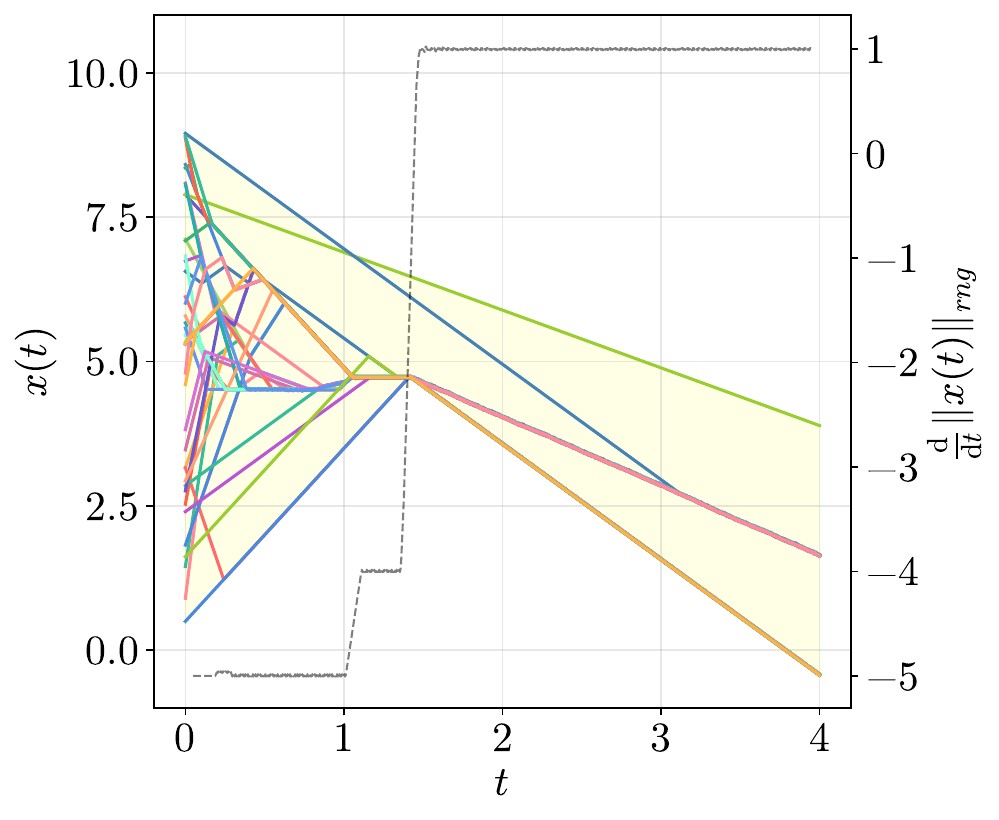}}
    \caption{Graph structure and simulation results under various conditions in \Cref{ss:s:2}. Simulations used the Simple Euler method with a time step of $2^{-10}$ s. Both consensus (\Cref{fig ss:s:2 2}) and dissensus solutions (\Cref{fig ss:s:2 3}) are shown, indicating that the system exhibits Weak Consensus. The yellow region spans from $m(x(t))$ to $M(x(t))$; the gray dashed line shows $\frac{\mathrm{d}}{\mathrm{d}t}\D{x(t)}$, i.e., the rate at which this region expands. It never exceeds $\A(\mathcal{W}) = 1$, and the cases in \Cref{fig ss:s:2 3} reach it at certain times.}
    \label{fig ss:s:2}
\end{figure}

This example shows that if $\A(\mathcal{W}) > 0$, both consensus and dissensus solutions may arise, and the derivative of the difference between the maximum and minimum agent states $\frac{\mathrm{d}}{\mathrm{d}t}\D{x(t)}$ can reach $\A(\mathcal{W})$. The outcome depends on the initial condition and is generally unpredictable.

Consider the case $n = 40$, with a strongly connected graph $\mathcal{W}$ and disturbance bound illustrated in \Cref{fig ss:s:2 1}. In this figure, the self-loop at node $i$ represents the upper bound of the disturbance term acting on agent $i$, and edge color intensity indicates weight magnitude, with all weights being integers between $1$ and $9$.

We consider the worst-case disturbance $d(t,x)$, which maximally opposes consensus by applying $w_{ii}$ to nodes near the maximum and $-w_{ii}$ to those near the minimum:
\begin{equation}
    d_i(t,x) = \begin{cases}
        w_{ii}& x_i \ge M(x)-0.01\\
        -w_{ii}& x_i \le m(x)+0.01\  \text{and}\  x_i < M(x)-0.01\\
        0&\text{otherwise}\\
    \end{cases}
\end{equation}

Here, $\A(\mathcal{W}) = 1 > 0$, and the simulation results are shown in \Cref{fig ss:s:2} with details in the caption.

\subsection{Example of $\A(\mathcal{W}) < 0$ with Random Disturbance
}\label{ss:s:3}

\begin{figure}[htbp]
    \centering  
    \hspace{-0.08\linewidth}
    \subfigure[Graph structure. Darker edge colors indicate higher edge weights.]{\label{fig ss:s:3 1}\includegraphics[width=0.47\linewidth]{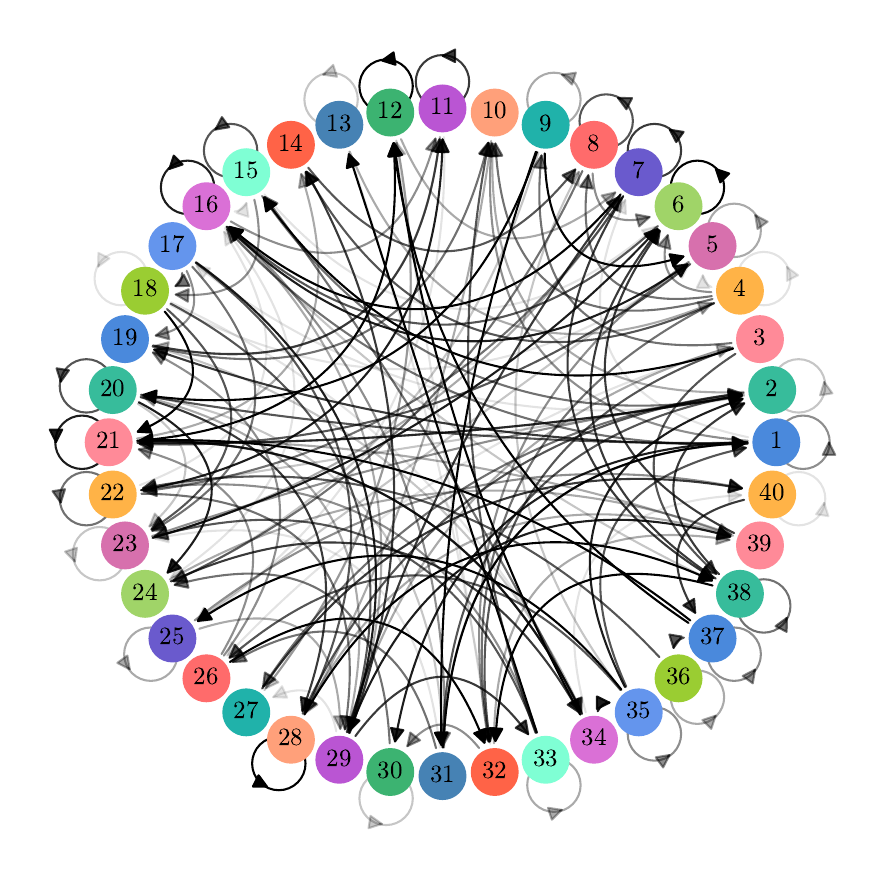}}
    \subfigure[Plots of the first seven disturbance functions over time.]{\raisebox{0.06\height}{\label{fig ss:s:3 2}\includegraphics[width=0.401\linewidth]{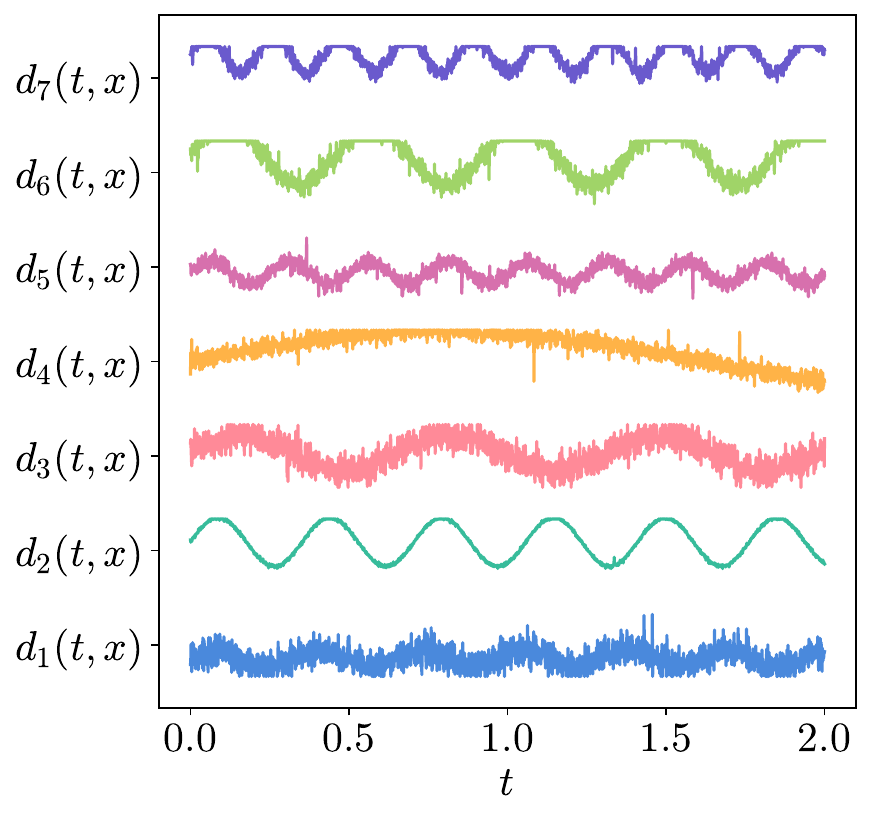}}}
    \subfigure[Simulation with two initial values]{\label{fig ss:s:3 3}\includegraphics[width=0.485\linewidth]{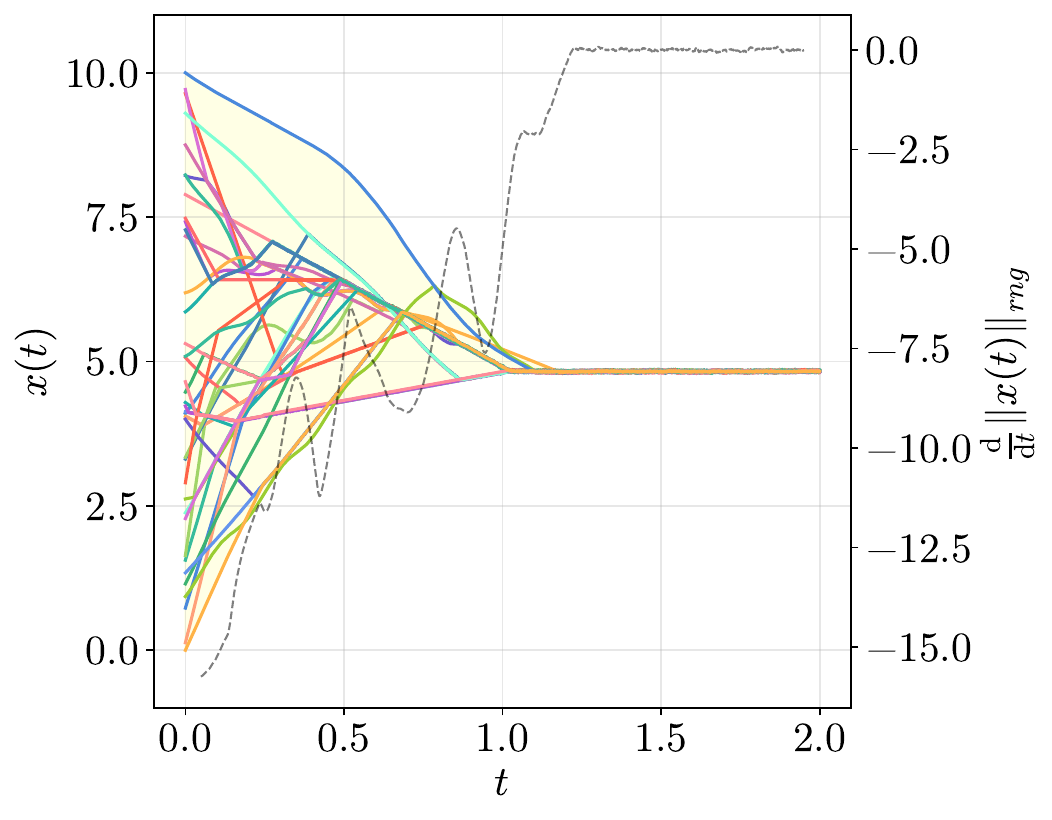}\includegraphics[width=0.47\linewidth]{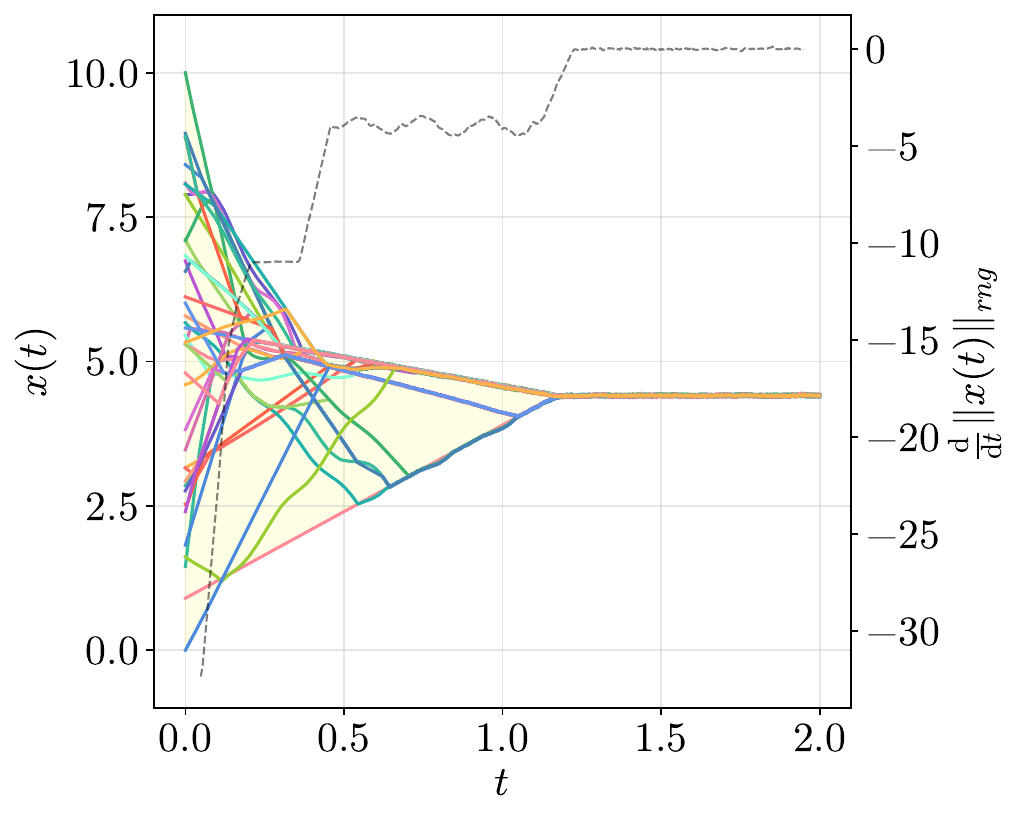}}
    \caption{Graph structure, disturbance functions, and simulation results in \Cref{ss:s:3}. Simulations used the Simple Euler method with a time step of $2^{-10}$ s. Both runs show that the system reaches consensus before $T(\D{x(0)}) = 10$ s despite disturbances. The yellow region spans from $m(x(t))$ to $M(x(t))$, and the gray dashed line shows $\frac{\mathrm{d}}{\mathrm{d}t}\D{x(t)}$, i.e., the rate at which this region expands. This rate stays below $\A(\mathcal{W}) = -1$ (i.e., the contraction rate is at least $1$) whenever $\D{x(t)} \ne 0$. (The moving average smooths out sudden changes in $\frac{\mathrm{d}}{\mathrm{d}t}\D{x(t)}$, causing sharp jumps to appear as gradual slopes in the plot.)}
    \label{fig ss:s:3}
\end{figure}

This example shows that if $\A(\mathcal{W}) < 0$, the system always reaches consensus, with $\frac{\mathrm{d}}{\mathrm{d}t}\D{x(t)} < \A(\mathcal{W})$ at all times. Consensus is guaranteed before time $T(\D{x(0)})$ where $T(\cdot)$ is defined in \Cref{eq in t:m:4}.

Consider the case $n = 40$, with a strongly connected graph $\mathcal{W}$ and disturbance bound illustrated in \Cref{fig ss:s:3 1}. In this figure, the self-loop at node $i$ represents the upper bound of the disturbance term acting on agent $i$, and edge color intensity indicates weight magnitude, with all weights being integers between $1$ and $9$.

We model $d_i(t, x)$ as a time-varying, state-independent random disturbance composed of white noise, low-frequency harmonics, and impulses, and bounded within $[-w_{ii}, w_{ii}]$.

Here, $\A(\mathcal{W}) = -1 < 0$. Graph structure, the first seven components of $d(t,x)$, and simulation results are shown in \Cref{fig ss:s:3} with details in the caption.

\section{Conclusion}\label{s:Conclusion}
We investigated the signum consensus protocol for continuous-time multi-agent systems over weighted directed graphs subject to bounded disturbances. Unexpected properties of this discontinuous consensus protocol are uncovered; as one of the simplest discontinuous protocols, its behavior may reflect that of more complex ones. We introduce the Polarization Index to capture the system’s tendency toward consensus or dissensus, establish the necessary and sufficient conditions for consensus, and provide a least upper bound on the consensus time. And we reformulate the computation of the Polarization Index as a mixed-integer programming program in certain cases, reducing average-case complexity. By grounding these findings in an opinion dynamics framework, we reveal an intrinsic connection between system dynamics, network topology, and the emergence of dissensus. This suggests that such links are not incidental but structurally embedded, offering rich directions for further exploration. 

Since in some systems (e.g., digital circuits) the computation of the signum function is exact, it is also worth studying the system under the simplest convex definition of Filippov solution.

\appendix

\crefalias{section}{appendix}

\section{Proof of \Cref{t:m:1}}\label{Proof of t:m:1}

Since $d_i(\cdot,\cdot)$ is a bounded single-valued function, \cite[Lemma 1, §6]{filippov2013differential} implies that the set-valued function $\mathcal{D}_i(t,x)$ is upper semicontinuous in $(t,x)$. Moreover, it is straightforward to show that for every $(t,x)\in\mathbb{R}^{n+1}$, $\mathcal{D}_i(t,x)$ is nonempty, bounded, closed, and convex.
    
An identical argument shows that each $\mathcal{U}_{ij}(x)$ is upper semicontinuous in $x$ and has nonempty, bounded, closed, and convex values for every $x\in\mathbb{R}^n$.

Hence, combining \cite[Lemma 16, §5, Lemma 2, §6]{filippov2013differential} with further derivations, the set-valued function \(\mathcal{F}(t,x)\) defined in \Cref{e:s:2} is upper semicontinuous in \((t, x)\) and, for all \((t, x) \in \mathbb{R}^{n+1}\), the set \(\mathcal{F}(t, x)\) is nonempty, bounded, closed, and convex.

Therefore, \cite[Theorem 1, §7]{filippov2013differential} guarantees that for any initial $(t^0,x^0)$ there is a local solution through $(t^0,x^0)$, and by the usual extension argument this solution extends to all of $\mathbb{R}$.

\section{Verification of Solutions in \Cref{Example 2}
}\label{Proof of Example 2}

We prove that both $x(\cdot)$ and $x'(\cdot)$ are solutions of \Cref{e:c:3} with the same initial condition $x(0) = [0,1,2]^T$. 

For $t\in [0,1/6)$, $x_1(t)$, $x_2(t)$ and $x_3(t)$ are different so $\mathcal{F}(t,x(t)) = \{[4,-2,-1]^T\}$ and $\dot{x}(t) \in \mathcal{F}(t,x(t))$ holds. 

For $t\in (1/6,1)$, $x_1(t) = x_2(t)\not = x_3(t)$ so $\mathcal{U}_{12}(x(t)) = [-2,2]$, $\mathcal{U}_{21}(x(t)) = [-3,3]$ (Here, $[-2,2]$ and $[-3,3]$ are closed interval, not vectors), and
\begin{equation}\label{eq appendix A 1}
        \mathcal{F}(t,x(t)) = \operatorname{conv}\{[-2,-2,-1]^T, [4,-2,-1]^T, [-2,2,-1]^T, [4,2,-1]^T\},
\end{equation}

Since $\dot{x}(t) = [2/5,2/5,-1]^T$, $\dot{x}(t) \in \mathcal{F}(t,x(t))$ holds.

For $t\in (1,\infty)$, $x_1(t) = x_2(t) = x_3(t)$ so
\begin{equation}
    \begin{aligned}
        \mathcal{F}(t,x(t)) = \operatorname{conv}\{&[-4,-2,-1]^T, [-4,-2,1]^T, [-4,2,-1]^T, [-4,2,1]^T, \\
        &[4,-2,-1]^T, [4,-2,1]^T, [4,2,-1]^T, [4,2,1]^T \},
    \end{aligned}
\end{equation} and $\dot{x}(t) = [0,0,0]^T\in \mathcal{F}(t,x(t))$.

Thus, $x(\cdot)$ satisfies the differential inclusion for almost all $t$. And since $x(\cdot)$ is absolutely continuous, it is a Filippov solution of \Cref{e:c:3}.

For $x'(\cdot)$, the case $t \in [0, 1/6)$ is identical. For $t \in (1/6, \infty)$, $x'_1(t) = x'_2(t) \neq x'_3(t)$, so $\mathcal{F}(t,x(t))$ remains as in \Cref{eq appendix A 1}, and $\dot{x}'(t) = [-2, -2, -1]^T \in \mathcal{F}(t,x(t))$. Hence, $x'(\cdot)$ is also a Filippov solution.

\section{Proof of \Cref{t:m:2}}\label{Proof of t:m:2}

We will prove the result for $\frac{\mathrm{d}}{\mathrm{d}t}M(x(t))$, as the proof for $\frac{\mathrm{d}}{\mathrm{d}t}m(x(t))$ follows similarly.

By \cite[Lemma 2]{usevitch2020resilient}, the function $M(\cdot)$ is both locally Lipschitz and regular on $\mathbb{R}^n$. Therefore, by \cite[Theorem 2.2]{shevitz2002lyapunov}, we have $\frac{\mathrm{d}}{\mathrm{d}t}M(x(t))\in \widetilde{\mathcal{L}}_{\mathcal{F}}M(x(t))$ a.e. on $t\in \mathbb{R}$, 
\begin{equation}
    \widetilde{\mathcal{L}}_{\mathcal{F}}M(x(t)) =  \{a\in \mathbb{R}\mid \exists v\in\mathcal{F}(t,x(t))  \ \text{s.t.} \ \forall\zeta\in \partial M(x(t)),\zeta^T v = a\}.
\end{equation}

By \cite[Lemma 3]{usevitch2020resilient}, each $z\in \partial M(x(t))$ can be written as the convex combination $z = E_{S_M(x(t))}\theta$, where $E_{S_M(x(t))}$ consists of the columns of the identity matrix $n\times n$ indexed by the set $S_M(x(t))$, and $\theta\in \mathbb{R}^{\operatorname{card}(S_M(x(t)))} $ satisfies $\theta \succeq \pmb{0}$ and $ \pmb{1}^T \theta = 1$. Therefore, it follows that $a\in \widetilde{\mathcal{L}}_{\mathcal{F}}M(x(t))$ if and only if there exists a $v\in \mathcal{F}(t,x(t))$ s.t. $(E_{S_M(x(t))}\theta)^Tv = \theta^T(E_{S_M(x(t))}^Tv) = a$ for all $\theta$ satisfying $\theta \succeq \pmb{0}$ and $\pmb{1}^T \theta = 1$. This condition holds if and only if $E_{S_M(x(t))}^Tv = a\pmb{1}$.
Therefore, we conclude that: 
\begin{equation}
    a\in \widetilde{\mathcal{L}}_{\mathcal{F}}M(x(t)) \iff \exists v\in\mathcal{F}(t,x(t)),v_{S_M(x(t))} = a \pmb{1}.
\end{equation}

Thus, we obtain the desired result:
\begin{equation}
    \frac{\mathrm{d}}{\mathrm{d}t}M(x(t)) \in \left(\operatorname{span}(\pmb{1})\cap\mathcal{F}_{S_M(x(t))}(t, x(t))\right)_1.
\end{equation}

(This proof partially cites the proof of Theorem 3 from \cite{usevitch2020resilient}.)

\section{Proof of \Cref{c f F subset S}}\label{Proof of c f F subset S}

We prove the case of $S_M$; the case of $S_m$ is similar.

For all $y\in \mathcal{F}(t,x)$, we have
\begin{equation}
    y_i \ge -\sum_{j\in \mathcal{V}}w_{ji} \quad \forall i\in S_M(x).
\end{equation}

Since $x_i > x_j$ for all $i \in S_M(x)$ and $j \notin S_M(x)$, and $\operatorname{card}(\mathcal{V})$ is finite, there exists $\epsilon > 0$ such that $x_i - \epsilon > x_j$ for all such pairs. Consequently, for all $i \in S_M(x)$ and $j \notin S_M(x)$, we have $\mathcal{U}_{ji}(x) = \{-w_{ji}\}$. Therefore, 
\begin{equation}
\begin{aligned}
    y_i &\le \sum_{j\in S_M(x)}w_{ji} - \sum_{j\notin S_M(x)}w_{ji} \quad \forall i\in S_M(x).
\end{aligned}
\end{equation}

Therefore, $y_{S_M(x)}\in S(S_M(x))$.

\section{Proof of \Cref{t:m:3}}\label{Proof of t:m:3}

To simplify notation, we define $\phi: \mathcal{P}(\mathcal{V}) \to \{-1, 1\}^n$ as follows:
\begin{equation}\label{def for phi}
    \phi_i(V) = \begin{cases}
        1 & i\in V,\\
        -1 & i\notin V.
    \end{cases}
\end{equation}

Hence, $\alpha^V$ and $\beta^V$ are equivalent to:
\begin{equation}
    \alpha^V = -(\mathcal{W}^{T}\pmb{1})_V, \qquad
    \beta^V = \left(\mathcal{W}^T\phi(V)\right)_V.
\end{equation}

We will need the following Lemmas for the proof of \Cref{t:m:3}.

\begin{Lemma}\label{l in proof t:m:3}
    Let $f:\mathbb{R} \to \mathbb{R}_{\ge 0}$ be absolutely continuous and satisfy 
    \begin{equation}
        \frac{\mathrm{d}}{\mathrm{d}t}f(t) \le A \quad \text{a.e.\ on }\{t\in \mathbb{R}\mid f(t)>0\},
    \end{equation}
    where $A\in \mathbb{R}$. Then for all $t^0\in \mathbb{R}$
    \begin{equation}
        f(t) \le max\{0,f(t^0)+At\} \quad \forall t \ge t^0.
    \end{equation}
\end{Lemma}
\begin{proof}
Without loss of generality, assume $t^0 = 0$.

For $t=0$, the statement is trivial. 

Suppose, to the contrary, that there exists \(\tau>0\) such that
\begin{equation}
    f(\tau) > \max\{0,f(0)+A\tau\}.
\end{equation}

Clearly, $f(\tau) > 0$. Assume that $f(t) = 0$ for some $t \in [0, \tau]$, set 
\begin{equation}
    \tau' = \operatorname{sup}\{t\in[0,\tau]\mid f(t) = 0\}.
\end{equation} 

By continuity of $f(\cdot)$, $f(\tau') = 0$; otherwise, if $f(\tau') > 0$, then $f(t) > 0$ for $t$ near $\tau'$, contradicting the fact that $\tau'$ is the supremum. $f(t) > 0$ for all $t \in (\tau', \tau]$, and for any such $t$ we have 
\begin{equation}\label{eq in l in proof of t:m:3}
        f(t) = f(\tau)+\int_\tau^t\frac{\mathrm{d}}{\mathrm{d}s}f(s)\mathrm{d}s \ge f(\tau)+\int_\tau^tA\mathrm{d}s = f(\tau)+A(t-\tau).
\end{equation}

Case i: $A \le 0$. $\lim_{t\to \tau'^+}f(t) \ge f(\tau) + A(\tau'-\tau) \ge f(\tau) > 0$ but $f(\tau') = 0$ which contradicts the continuity of $f(\cdot)$.

Case ii: $A > 0$. Then 
\begin{equation}\label{eq 2 in l in proof of t:m:3}
    \max\{0,f(0)+At\} = f(0)+At \ge 0,
\end{equation} for all $t\in \mathbb{R}_{\ge 0}$. However, $0 = \lim_{t\to \tau'^+}f(t) \ge f(\tau)+A(\tau'-\tau) > f(0)+A\tau+A(\tau'-\tau) = f(0)+A\tau'$ which contradicts \Cref{eq 2 in l in proof of t:m:3}. 

Therefore, $f(t) > 0$ for $t \in [0, \tau]$ and \Cref{eq in l in proof of t:m:3} holds on $[0,\tau]$. Consequently,
\begin{equation}
        f(0) \ge f(\tau)-A\tau > \max\{0,f(0)+A\tau\}-A\tau \ge f(0)+A\tau-A\tau = f(0).
\end{equation} 

Thus, we obtain \( f(0) > f(0) \), a contradiction.  

Therefore for all $t\ge0$, $f(t) \le max\{0,f(0)+At\}$.
\end{proof}

\begin{Lemma}\label{l 2 in proof of t:m:3}
    Let $x,a^1,b^1,a^2,b^2\in \mathbb{R}^n$ where $a^1 \preceq b^1$, $a^2 \preceq b^2$ and $a^1+a^2\preceq x\preceq b^1+b^2$, then there must exist $a^1\preceq x^1 \preceq b^1$ and $a^2 \preceq x^2 \preceq b^2$ such that $x = x^1+x^2$.
\end{Lemma}
\begin{proof}
    Define
    \begin{equation}
        R(a, b) := \{x \in \mathbb{R}^n \mid a \preceq x \preceq b\}.
    \end{equation}
    
    The set $R(a, b)$ is an axis-aligned hyperrectangle in $\mathbb{R}^n$. Clearly,
    \begin{equation}
        R(a, b) = \operatorname{conv} P(a, b),
    \end{equation}
    where
    \begin{equation}
        P(a, b) := \{c \mid c_j \in \{a_j, b_j\},\ \forall j \in \{1, 2, \dots, n\}\}
    \end{equation}
    is the set of vertices of the hyperrectangle.

    Moreover, we have
    \begin{equation}
        P(a^1_j + a^2_j,\, b^1_j + b^2_j) \subset P(a^1_j, b^1_j) + P(a^2_j, b^2_j),
    \end{equation}
    where $+$ denotes the Minkowski sum.

    By a property of Minkowski sums: For all non-empty subsets $S_1, S_2$ of a real vector space, the convex hull of their Minkowski sum is equal to the Minkowski sum of their convex hulls. we obtain:
    \begin{equation}
        \begin{aligned}
            R(a^1_j + a^2_j,\, b^1_j + b^2_j) &= \operatorname{conv} P(a^1_j + a^2_j,\, b^1_j + b^2_j) \\
            &\subset \operatorname{conv} [P(a^1_j, b^1_j) +  P(a^2_j, b^2_j)]
            \\
            &= \operatorname{conv} P(a^1_j, b^1_j) + \operatorname{conv} P(a^2_j, b^2_j) \\
            &= R(a^1_j, b^1_j) + R(a^2_j, b^2_j).
        \end{aligned}
    \end{equation}

Therefore, for any $x \in R(a^1_j + a^2_j,\, b^1_j + b^2_j)$, there exist $x^1 \in R(a^1_j, b^1_j)$ and $x^2 \in R(a^2_j, b^2_j)$ such that $x = x^1 + x^2$. This completes the proof.
\end{proof}

\begin{Lemma}\label{l 3 in proof of t:m:3}
    Let $\mathcal{W} \in \mathbb{R}_{\ge 0}^{n \times n}$, and $V \subsetneq \mathcal{V}$ be a nonempty proper subset. Then for any $y \in S(V)$, there exists a constant vector $d'_V \in \mathbb{R}^{\operatorname{card}(V)}$ with each component satisfying $d'_i \in [-w_{ii}, w_{ii}]$ for all $i \in V$, such that if a disturbance $d(\cdot,\cdot)$ is fixed on $V$ as $d_V(t,x)\equiv d'_V$, then
    \begin{equation}
    y \in \mathcal{F}_V(t,x) \quad \text{(resp. } -y \in \mathcal{F}_V(t,x) \text{)}
    \end{equation}
    holds for all $x$ satisfying $V \subset S_M(x)$ (resp. $V \subset S_m(x)$).
\end{Lemma}
\begin{proof}
    We prove the case of $S_M$; the case of $S_m$ is similar.

    Let $w^{diag} = [w_{11},w_{22},\dots,w_{nn}]^T$ and $\mathcal{W}' := \mathcal{W}-w^{diag}$ then 
    \begin{equation}
        \alpha^V = -(\mathcal{W}'^T\pmb{1})_V - w^{diag}_V, \qquad
        \beta^V = \left(\mathcal{W}'^T\phi(V)\right)_V + w^{diag}_V.
    \end{equation}
    
    Since 
    \begin{equation}
    \begin{aligned}
        -(\mathcal{W}'^T\pmb{1})_V \preceq \left(\mathcal{W}'^T\phi(V)\right)_V,\qquad
    -w^{diag}_V \preceq w^{diag}_V,
    \end{aligned}
    \end{equation}
    it follows from \Cref{l 2 in proof of t:m:3} that for any $y \in S(V)$, there exist $y',d'_V \in \mathbb{R}^{\operatorname{card}(V)}$ satisfying
    \begin{equation}
            -(\mathcal{W}'^T\pmb{1})_V \preceq y' \preceq \left(\mathcal{W}'^T\phi(V)\right)_V,\qquad
    -w^{diag}_V \preceq d'_V \preceq w^{diag}_V,
    \end{equation}
    such that $y = y' + d'_V$.

    Let $d_V(t,x) \equiv d'_V$. 

    Since $x_i > x_j$ for all $i \in S_M(x)$ and $j \notin S_M(x)$, and $\operatorname{card}(\mathcal{V})$ is finite, there exists $\epsilon > 0$ such that $x_i - \epsilon > x_j$ for all such pairs. Consequently, for all $i \in S_M(x)$ and $j \notin S_M(x)$, we have $\mathcal{U}_{ji}(x) = \{-w_{ji}\}$. And for $i,j \in S_M(x)$ we have $\mathcal{U}_{ji}(x) = [-w_{ji},w_{ji}]$. Therefore, 
    \begin{equation}
        \mathcal{F}_V(t,x) = d'_V+\bigl\{y\in\mathbb{R}^{\operatorname{card}(V)}\mid-(\mathcal{W}'^T\pmb{1})_V\preceq y\preceq\bigl(\mathcal{W}'^T\phi(S_M(x))\bigr)_V\bigr\}
    \end{equation}
    for all $x$ satisfying $V \subset S_M(x)$. Since 
    \begin{equation}
        \left(\mathcal{W}'^T\phi(V)\right)_V\preceq \left(\mathcal{W}'^T\phi(S_M(x))\right)_V,
    \end{equation}
    we have $-(\mathcal{W}'^T\pmb{1})_V\preceq y'\preceq\bigl(\mathcal{W}'^T\phi(S_M(x))\bigr)_V$ and $y \in \mathcal{F}_V(t, x)$ for all $x$ satisfying $V \subset S_M(x)$. 
\end{proof}

We now give the proof of \Cref{t:m:3}. By \Cref{t:m:2}, for any solution $x:\mathbb{R}\to \mathbb{R}^n$ of \Cref{e:c:3}, the derivatives $\frac{\mathrm{d}}{\mathrm{d}t}M(x(t))$ and $\frac{\mathrm{d}}{\mathrm{d}t}m(x(t))$ exist and satisfy \Cref{e:m:1} a.e. on $t\in \mathbb{R}$. Then by \Cref{c f F subset S},
\begin{equation}
\begin{aligned}
    \frac{\mathrm{d}}{\mathrm{d}t}M(x(t)) & \le  \sup\left( \operatorname{span}(\pmb{1}) \cap \mathcal{F}_{S_M}(t,x(t))\right)_1 \\
    & \le  \sup\left( \operatorname{span}(\pmb{1}) \cap S(S_M(x(t)))\right)_1 = \Q(S_M(x(t))),
\end{aligned}
\end{equation}
a.e. on $t\in \mathbb{R}$. The same is for $-\frac{\mathrm{d}}{\mathrm{d}t}m(x(t)) \le \Q(S_m(x(t)))$ and,
\begin{equation}
    \begin{aligned}
        \frac{\mathrm{d}}{\mathrm{d}t}\D{x(t)} & = \frac{\mathrm{d}}{\mathrm{d}t}M(x(t)) - \frac{\mathrm{d}}{\mathrm{d}t}m(x(t)) \le \Q(S_M(x(t)))+\Q(S_m(x(t))) \le \A(\mathcal{W}).
    \end{aligned}
\end{equation}
a.e. on $t\in \{t\in \mathbb{R}_{\ge 0}\mid \D{x(t)}>0\}$.

Therefore by \Cref{l in proof t:m:3}, for all $t^0\in \mathbb{R}$ 
\begin{equation}
    \D{x(t)} \le \max\Big\{0,\D{x(t^0)}+\A(\mathcal{W})t\Big\} \quad \forall t \ge t^0.
\end{equation}

We now show that for every $r\in \mathbb{R}_{\ge 0}$, there exists a solution $x'(\cdot)$ of \Cref{e:c:3} with $d'(\cdot,\cdot)$ satisfying \Cref{d satisfies} and $\D{x'(0)} = r$ such that $x'(\cdot)$ satisfies \Cref{eq 2 in t:m:3}. We prove the case $r = 2$; other cases follow similarly.

For $t \in \mathbb{R}$, if $\D{x'(0)} + \A(\mathcal{W})t < 0$, then by the previous analysis, we must have
\begin{equation}
    \D{x'(t)} = 0 = \max\Big\{0,\D{x'(0)}+\A(\mathcal{W})t\Big\}.
\end{equation}

Therefore, we only need to consider the case $t \in T$ where
\begin{equation}
    T := \big\{t \in \mathbb{R}_{\ge 0} \mid \D{x'(0)}+\A(\mathcal{W})t\ge 0\big\}.
\end{equation}

Suppose that disjoint nonempty subsets $V_1, V_2 \subset \mathcal{V}$ satisfy $\A(\mathcal{W}) = \Q(V_1) + \Q(V_2)$, and let $V_3 = (V_1 \cup V_2)^c$. Consider the initial condition
\begin{equation}\label{eq 3 in proof of t:m:3}
    x'_{V_1}(0) = \pmb{1}, \quad
    x'_{V_2}(0) = -\pmb{1}, \quad
    x'_{V_3}(0) = \pmb{0}.
\end{equation}

Since $\Q(V)\pmb{1} \in S(V)$ for all nonempty $V \subset \mathcal{V}$, it follows from \Cref{l 3 in proof of t:m:3} that there exists a function $d'(\cdot,\cdot)$ satisfying \Cref{d satisfies}, such that
\begin{equation}
    \Q(V_1)\pmb{1} \in \mathcal{F}_{V_1}(t, x) \quad \text{and} \quad -\Q(V_2)\pmb{1} \in \mathcal{F}_{V_2}(t, x)
\end{equation}
for any $x$ satisfying $V_1 \subset S_M(x)$ and $V_2 \subset S_m(x)$.

Therefore if
\begin{equation}
    -1-\Q(V_2)t \le x'_j(t) \le 1+\Q(V_1)t, \quad \forall t\in T,j\in V_3,
\end{equation}
then 
\begin{equation}
    x'_j(t) =\begin{cases}
        1+\Q(V_1)t & j\in V_1,\\
        -1-\Q(V_2)t & j\in V_2,
    \end{cases} 
\end{equation} satisfies the differential inclusion $\dot{x}\in \mathcal{F}(t,x)$ with $d'(\cdot,\cdot)$ for $t\in T$.

Consider the subgraph $\mathcal{G}_3 = (V_3, \mathcal{E}_3, \mathcal{W}_3)$ induced by $V_3$, and $d'':\mathbb{R}\times \mathbb{R}^{\operatorname{card}(V_3)}\to \mathbb{R}^{\operatorname{card}(V_3)}$ defined as 
\begin{equation}
    d''_j(t,x) = d'_j(t,x)+\begin{cases}
        -\sum_{i\in V_1\cup V_2} w_{ij} & 1+\Q(V_1)t < x_j,\\
        \sum_{i\in V_1}w_{ij}-\sum_{i\in V_2} w_{ij} & -1-\Q(V_2)t \le x_j \le 1+\Q(V_1)t,\\
        \sum_{i\in V_1\cup V_2} w_{ij} & x_j < -1-\Q(V_2)t,
    \end{cases}
\end{equation}
for all $j \in  V_3$.

This subsystem essentially describes the behavior of agents in $V_3$ when the agents in $V_1$ and $V_2$ remain on $1 + \Q(V_1)t$ and $-1 - \Q(V_2)t$, respectively. The function $d''(\cdot,\cdot)$ includes the influence of agents in $V_1$ and $V_2$ on those in $V_3$.

By \Cref{t:m:1}, the subsystem admits a solution $x''(\cdot)$. We now show that every solution $x''(\cdot)$ of the subsystem \Cref{e:c:3} of $\mathcal{G}_3$ with $d''(\cdot, \cdot)$ satisfies
\begin{equation}\label{eq for subsys in proof of t:m:3}
    -1-\Q(V_2)t \le x''_j(t) \le 1+\Q(V_1)t, \quad \forall t\in T,j\in V_3.
\end{equation}

We prove that \Cref{eq for subsys in proof of t:m:3} holds by contradiction. Suppose that there exists some time $\tau$ such that \Cref{eq for subsys in proof of t:m:3} fails. We consider the case where some node in $V_3$ exceeds the upper bound, since the case of falling below the lower bound is analogous. That is
\begin{equation}
   M(x''(\tau))  > 1 + \Q(V_1)\tau.
\end{equation}

Define
\begin{equation}
    \tau' := \sup\big\{t \in [0, \tau] \mid M(x''(t)) \le 1+\Q(V_1)t\big\},
\end{equation}
by continuity of $M(\cdot)$ and $x''(\cdot)$, $M(x''(\tau')) = 1 + \Q(V_1)\tau'$. 

Then there exist a subset $T' \subset [\tau', \tau]$ of nonzero measure where $\frac{\mathrm{d}}{\mathrm{d}t}M(x''(t)) > \Q(V_1)$ for all $t \in T'$. There must exists a subset $V \subset V_3$ such that $\Q(V) > \Q(V_1)$, which implies $\Q(V) + \Q(V_2) > \Q(V_1) + \Q(V_2) = \A(\mathcal{W})$, contradicting the definition of $\A(\mathcal{W})$.

Therefore, 
\begin{equation}
    x'_j(t) =\begin{cases}
        1+\Q(V_1)t & j\in V_1,t\in T,\\
        -1-\Q(V_2)t & j\in V_2 ,t\in T,\\
        x''_j(t) & j\in V_3,t\in T,\\
        \frac{\Q(V_2)-\Q(V_1)}{\Q(V_1)+\Q(V_2)} & j\in\mathcal{V}, t\not\in T,
    \end{cases}
\end{equation}
is a solution of the system and 
\begin{equation}
    \D{x'(t)} = \max\Big\{0,\D{x'(0)}+\A(\mathcal{W})t\Big\},
\end{equation}
for all $t\in\mathbb{R}_{\ge 0}$.

\section{Proof of \Cref{t:m:5}}\label{Proof of t:m:5}

We will need the following Lemmas for the proof of \Cref{t:m:5}.

Define $\Qu,\Ql: \mathcal{P}(\mathcal{V})\setminus\{\emptyset\}\to \mathbb{R}$ as 
\begin{equation}\label{def of Ql}
    \Qu(V) := \min_{i\in V}\beta^V_i,\qquad \Ql(V) := \max_{i\in V}\alpha^V_i.
\end{equation}

The definition of $\Qu$ is consistent with \Cref{eq in t:m:5}, since by \Cref{def of alpha beta} we have $\beta^V_i = \sum_{j\in V} w_{ji} -\sum_{j\in \mathcal{V}\backslash V} w_{ji}$

\begin{Lemma}\label{l f Q least than Qu}
For all nonempty subsets \( V \subset \mathcal{V} \), the inequality \( \Q(V) \leq \Qu(V) \) holds.
\end{Lemma}
\begin{proof}
Assume, for contradiction, that \( \Q(V) > \Qu(V) \).  By definition, \( \Q(V)\pmb{1} \in S(V) \) and $\Q(V)\pmb{1} \preceq \beta^V$. Let \( i \in V \) be one of the indices where \( \beta^V_i = \Qu(V) \). Then, \( (\Q(V)\pmb{1})_i = \Q(V) > \Qu(V) = \beta^V_i \), which contradicts \( \Q(V)\pmb{1} \preceq \beta^V \), thereby proving the lemma.
\end{proof}

\begin{Lemma}\label{l f Q inf iif Ql Qu}
    For any nonempty subset \( V \subset \mathcal{V} \), the following are equivalent:
    \begin{equation}
        \Ql(V) \le \Qu(V) 
        \;\Longleftrightarrow\;
        \Q(V) = \Qu(V) 
        \;\Longleftrightarrow\;
        \Q(V) \ne -\infty.
    \end{equation}
\end{Lemma}
\begin{proof}
    Suppose \( \Ql(V) \le \Qu(V) \). Then
    \begin{equation}
        \alpha^V \preceq \Ql(V)\pmb{1} \preceq \Qu(V)\pmb{1} \preceq \beta^V,
    \end{equation}
    which implies \( \Qu(V)\pmb{1} \in \operatorname{span}(\pmb{1}) \cap S(V) \), and \( \Q(V) \ge \Qu(V) \). Combining with $\Q(V) \le \Qu(V)$ from \Cref{l f Q least than Qu} gives $\Q(V) = \Qu(V)$.
    
    Since \( \Qu(V) \in \mathbb{R} \), the implication \( \Q(V) = \Qu(V) \Rightarrow \Q(V) \ne -\infty \) is immediate.
    
    Suppose \( \Q(V) \ne -\infty \). Then there exists some scalar \( a \in \mathbb{R} \) such that \( a\pmb{1} \in S(V) \), i.e., 
    \begin{equation}
        \alpha^V_i \le a \quad \forall i \in V,\qquad a \le \beta^V_i \quad \forall i \in V.
    \end{equation}
    
    This implies \( \Ql(V) \le a \le \Qu(V) \), and hence \( \Ql(V) \le \Qu(V) \).
    
    Combining all parts, the three statements are equivalent.
\end{proof}

\begin{Lemma}\label{l f QuV1QuV2 QV1QV2}
    For any nonempty subset \( V\subset \mathcal{V} \), there exists a nonempty subset \( V' \subseteq V \) such that \( \Q(V') \geq \Qu(V) \).
\end{Lemma}
\begin{proof}
    If $\Q(V) \not= -\infty$, let $V' = V$ and by \Cref{l f Q inf iif Ql Qu} the lemma holds in this case.
    
    If $\Q(V) = -\infty$, then by \Cref{l f Q inf iif Ql Qu}, $\Ql(V) > \Qu(V)$. let $i\in V$ be one of the indices where $-\sum_{j\in \mathcal{V}} w_{ji} = \Ql(V)$, and let $V' = \{i\}$. Then 
    \begin{equation}
    \begin{aligned}
        \Q(V') &= \sup\left( \operatorname{span}(\pmb{1}) \cap S(V')\right)_1 = \operatorname{sup}S(V') = \sum_{j\in V'} w_{ji} -\sum_{j\in \mathcal{V}\backslash V'} w_{ji} \\
        &\ge -\sum_{j\in \mathcal{V}} w_{ji} = \Ql(V) > \Qu(V).
    \end{aligned}
    \end{equation}
\end{proof}

With the above lemmas, we now prove \Cref{t:m:5}. By \Cref{l f Q least than Qu}, we have
\begin{equation}
    \max_{\substack{V_1,V_2\subset\mathcal{V}\\ 
    V_1,V_2\neq\emptyset,\ V_1\cap V_2=\emptyset}} 
    \bigl[\Q(V_1)+\Q(V_2)\bigr] 
    \le\quad\
\max_{\substack{V_1,V_2\subset\mathcal{V}\\ 
    V_1,V_2\neq\emptyset,\ V_1\cap V_2=\emptyset}} 
    \bigl[\Qu(V_1)+\Qu(V_2)\bigr],
\end{equation}
while \Cref{l f QuV1QuV2 QV1QV2} establishes the reverse inequality.
Therefore,
\begin{equation}
    \A(\mathcal{W}) = \max_{\substack{V_1,V_2\subset\mathcal{V}\\V_1,V_2\neq\emptyset,\ V_1\cap V_2=\emptyset}}[\Qu(V_1)+\Qu(V_2)].
\end{equation}

\section{Proof of \Cref{t:m:6}}\label{Proof of t:m:6}

We will need the following Lemma for the proof of \Cref{t:m:6}. Some notations are defined in \Cref{Proof of t:m:5}.

\begin{Lemma}\label{l f i in V QuV le Quv}
    Let $V \subset \mathcal{V}$ be a nonempty subset, and let $i \in V$ be one of the indices where  $\beta^V_i = \Qu(V)$. Then for any subset $V' \subset V$ with $i \in V'$, $\Qu(V') \le \Qu(V)$ holds.
\end{Lemma}
\begin{proof}
    \begin{equation}
        \begin{aligned}
            \Qu(V') &\le \sum_{j\in V'} w_{ji} -\sum_{j\in \mathcal{V}\backslash V'} w_{ji} \le \sum_{j\in V'} w_{ji} -\sum_{j\in \mathcal{V}\backslash V'} w_{ji} + 2\sum_{j\in V\backslash V'} w_{ji}\\
            &= \sum_{j\in V} w_{ji} -\sum_{j\in \mathcal{V}\backslash V} w_{ji} =  \Qu(V),
        \end{aligned}
    \end{equation}
    and the lemma is proved.
\end{proof}

We now give the proof of \Cref{t:m:6}. Since $\A(\mathcal{W}) \ge z^*$ is straightforward, we aim to show that the equality holds if $z^* \le 0$.

We proceed by mathematical induction. We prove that for all \( m \in \{0, \dots, n-2\} \), if for every pair of nonempty disjoint subsets \( V, V' \subset \mathcal{V} \) with \( \operatorname{card}(V) + \operatorname{card}(V') = n + 1 - m \), the inequality \( \Qu(V) + \Qu(V') \le z^* \) holds, then for every pair of nonempty disjoint subsets \( V, V' \subset \mathcal{V} \) with \( \operatorname{card}(V) + \operatorname{card}(V') = n - m \), the inequality \( \Qu(V) + \Qu(V') \le z^* \) also holds.

For the base case \( m = 0 \), the condition \( \operatorname{card}(V) + \operatorname{card}(V') = n \) implies that \( V \) and \( V' \) are complementary subsets. By the given condition, we have \( \Qu(V) + \Qu(V') \le z^* \).

Now, assume \( m \in \{1, \dots, n-2\} \), and the inductive hypothesis holds. Let \( V, V' \subset \mathcal{V} \) be a pair of nonempty disjoint subsets with \( \operatorname{card}(V) + \operatorname{card}(V') = n - m \). By the given condition, we have \( \Qu(V) + \Qu(V^c) \le z^* \). Let \( i \in V^c \) be one of the indices where \( \beta^{V^c}_i = \Qu(V^c) \).  

When \( i \in V' \), by \Cref{l f i in V QuV le Quv}, \( \Qu(V') \leq \Qu(V^c) \), and thus
\begin{equation}
    \Qu(V) + \Qu(V') \le \Qu(V) + \Qu(V^c) \le z^*
\end{equation} 
holds.

When \( i \notin V' \), by the inductive hypothesis, \( \Qu(\{i\}\cup V) + \Qu(V') \le z^* \) and \( \Qu(V) + \Qu(\{i\} \cup V') \le z^* \).

If $\Qu(\{i\}\cup V) \not= \beta^{\{i\}\cup V}_i$, then there exists some $j\in V$ such that
\begin{equation}
    \Qu(\{i\}\cup V) = \beta^{\{i\}\cup V}_j.
\end{equation}

By \Cref{l f i in V QuV le Quv}, we have $\Qu(V) \le \Qu(\{i\}\cup V)$, and hence 
\begin{equation}
    \Qu(V) + \Qu(V') \le \Qu(\{i\}\cup V) + \Qu(V') \le z^*
\end{equation} 
holds. 

If $\Qu(\{i\}\cup V) = \beta^{\{i\}\cup V}_i$, then
\begin{equation}
    \Qu(\{i\}\cup V) = \sum_{j\in \{i\}\cup V} w_{ji} - \sum_{j\in V^c\setminus \{i\}} w_{ji}.
\end{equation}

Since
\begin{equation}
    \Qu(V^c) = \beta^{V^c}_i = \sum_{j\in V^c} w_{ji} - \sum_{j\in V} w_{ji},
\end{equation}
it follows that 
\begin{equation}
    \begin{aligned}
    &\Qu(V)+\Qu(V') + 2w_{ii} =\Qu(V)+\Qu(V')\\
    &+\sum_{j\in V^c} w_{ji} - \sum_{j\in V} w_{ji} + \sum_{j\in \{i\}\cup V} w_{ji} - \sum_{j\in V^c\setminus \{i\}} w_{ji}\\
    &= \Qu(V)+\Qu(V^c)+\Qu(\{i\}\cup V)+\Qu(V') \\
    &\le 2z^*.
    \end{aligned}
\end{equation}

Since $w_{ii} \ge 0$ and $z^* \le 0$, we conclude that $\Qu(V)+\Qu(V') \le z^*$ holds. Hence, if $z^* \le 0$, we have $\Qu(V)+\Qu(V') \le z^*$ for every pair of nonempty disjoint subsets $V,V'\subset \mathcal{V}$, which implies $\A(\mathcal{W}) \le z^*$ and completes the proof.

\section{Proof of \Cref{t:m:7}}\label{Proof of t:m:7}

We relax the inner minimization problem of the original problem shown in \Cref{eq int op problem} by expanding the range of values for $b$ to $\mathbb{R}^n$, obtaining the following optimization problem:
\begin{equation}\label{eq mim op problem}
\begin{aligned}
    \min_{b \in \mathbb{R}^n} \quad & a^T \mathcal{W} b \\
    \text{subject to} \quad &\pmb{1}^T b = 0, \\
                            & a^T b = 2, \\
                            & a_i b_i \ge 0, \quad \forall i \in \mathcal{V}.
\end{aligned}
\end{equation}

Since the range of $b$ has been expanded, for all $a$, the optimal value of \Cref{eq mim op problem} is less than or equal to the optimal value of the original inner minimization problem. 

Define $V^+ := \{i \in \mathcal{V} \mid a_i > 0\}$ and $V^- := \{i \in \mathcal{V} \mid a_i < 0\}$. According to the constraints, we have
\begin{equation}
\begin{aligned}
    &\sum_{i \in V^+} b_i = 1 \quad \text{and}\quad \sum_{i \in V^-} b_i = -1,\\
    &b_i \geq 0 \quad \forall i \in V^+\quad \text{and}\quad b_i \leq 0 \quad \forall i \in V^-.
\end{aligned}
\end{equation}

Then
\begin{equation}
    \begin{aligned}
        a^T\mathcal{W}b &= \sum_{i\in V^+} (a^T\mathcal{W})_ib_i - \sum_{i\in V^-} (a^T\mathcal{W})_i|b_i| \ge \min_{i\in V^+}(a^T\mathcal{W})_i - \max_{i\in V^-}(a^T\mathcal{W})_i= a^T\mathcal{W}b^*,
    \end{aligned}
\end{equation}
where $b^*$ takes the value $1$ at an index where $\min_{i \in V^+} (a^T \mathcal{W})_i$ is attained, and takes the value $-1$ at an index where $\max_{i \in V^-} (a^T \mathcal{W})_i$ is attained, with all other entries being $0$.

Since $b^*$ is within the feasible region of the original inner problem, \Cref{eq mim op problem} and the original inner problem have the same optimal solution and optimal value.

For each $a$, \Cref{eq mim op problem} is a linear programming problem, and its dual problem is given by:
\begin{equation}\label{eq inner ld dual op problem}
    \begin{aligned}
        \max_{c\in \mathbb{R}^{n+2}} \quad & 2 c_2\\
        \text{subject to} \quad &\begin{bmatrix}
                            \pmb{1} & a & \operatorname{diag}(a)
                        \end{bmatrix} c = \mathcal{W}^Ta,\\
                        & c_i \ge 0, \quad \forall i\in \{3,4,\dots,n+2\}.
    \end{aligned}
\end{equation}

For $a \in \{-1, 1\}^n$ that is neither all positive nor all negative, \Cref{eq mim op problem} has a feasible solution. Therefore, the optimal value of the dual problem \Cref{eq inner ld dual op problem} is the same as the optimal value of \Cref{eq mim op problem}. In conclusion, the optimal value of \Cref{eq int op problem} is equal to the optimal value of \Cref{eq ld dual op problem}.

\bibliographystyle{elsarticle-num}
\bibliography{references}

\end{document}